\documentclass[reqno,a4paper]{amsart}
\usepackage[pagebackref]{hyperref}
\usepackage{mathtools}
\usepackage{enumerate}
\usepackage[left=2.5cm,right=2.5cm]{geometry}
\allowdisplaybreaks

\newtheorem{theorem}{Theorem}[section]
\newtheorem{lemma}[theorem]{Lemma}
\newtheorem{corollary}[theorem]{Corollary} 
\newtheorem{prop}[theorem]{Proposition} 

\theoremstyle{definition}
\newtheorem{defi}[theorem]{Definition}

\theoremstyle{remark}
\newtheorem{remark}[theorem]{Remark}

\numberwithin{equation}{section}

\newcommand{\I}{\mathbf{1}}
\newcommand{\ind}[1]{1_{#1}} 
\newcommand{\E}{\mathbb{E}} 
\newcommand{\p}{\mathbb{P}} 
\newcommand{\R}{\mathbb{R}} 
\newcommand{\C}{\mathbb{C}}
\newcommand{\Z}{\mathbb{Z}}
\newcommand{\N}{\mathbb{N}}
\newcommand{\spec}{\mathbf{Sp}}

\DeclareMathOperator{\Med}{Med}	
\DeclareMathOperator{\Var}{Var}	
\DeclareMathOperator{\Cov}{Cov}	
\DeclareMathOperator{\tr}{tr} 	
\DeclareMathOperator{\id}{Id} 	
\DeclareMathOperator{\D}{{\rm d_{L}}} 	

\author{Rados{\l}aw Adamczak} %
\address[RA]{Institute of Mathematics, University of Warsaw \& Institute of Mathematics of the Polish Academy of Sciences} %
\email{R.Adamczak@mimuw.edu.pl}

\thanks{Research partially supported by the National Science Centre, Poland, grants
no. 2015/18/E/ST1/00214}

\title[Random non-Abelian $G$-circulants]{Random non-Abelian $G$-circulant matrices. Spectrum of random convolution operators on large finite groups}

\begin{document}
\maketitle

\begin{abstract}
We analyse the limiting behavior of the eigenvalue and singular value distribution for random convolution operators on large (not necessarily Abelian) groups, extending the results by M. Meckes for the Abelian case. We show that for regular sequences of groups the limiting distribution of eigenvalues (resp. singular values) is a mixture of eigenvalue (resp. singular value) distributions of Ginibre matrices with the directing measure being related to the limiting behavior of the Plancherel measure of the sequence of groups. In particular for the sequence of symmetric groups, the limiting distributions are just the circular and quarter circular laws, whereas e.g. for the dihedral groups the limiting distributions have unbounded supports but are different than in the Abelian case.

We also prove that under additional assumptions on the sequence of groups (in particular for symmetric groups of increasing order) families of stochastically independent random projection operators converge in moments to free circular elements.

Finally, in the Gaussian case we provide Central Limit Theorems for linear eigenvalue statistics.
\end{abstract}

\section{Introduction}

Consider a large finite group $G$, a function $X\colon G \to \C$ and the convolution operator $P_X \colon \C^G \to \C^G$, defined with the formula
\begin{displaymath}
(P_X v)(h) = (X\ast v)(h) = \sum_{g \in G} X(h g^{-1})v(g)
\end{displaymath}
for $v \in \C^G$, $h \in G$.

In what follows we will often identify operators on $\C^G$ with their matrices in the standard basis (i.e. the basis given by functions $\delta_g$ assigning one to $g$ and zero to other elements of $G$). In particular, with this identification $P_X = [X(hg^{-1})]_{h,g\in G}$. Following Diaconis \cite{MR1059483} we will call such matrices $G$-circulant matrices (if $G$ is the cyclic group $\Z_N$, we recover the usual notion of circulant matrices, see e.g. \cite{MR543191}).

The study of $P_X$ is a classical topic in harmonic analysis. In the last fifteen years considerable attention in Random Matrix Theory community has been devoted to investigation of asymptotic spectral properties of $P_X$ where $X$ is a random sequence and $G = \Z_N$ ($N\to \infty$) or more generally $G$ is a large Abelian group (see e.g. \cite{MR2682263,MR2795050,MR2797949,MR3069372}). The interest in this topic has been related to the general idea of Bai \cite{MR1711663} of studying random matrices with special type of dependence between entries, in particular matrices with additional linear or combinatorial structure, which limits the amount of stochastic independence with respect to the classical theory of random matrices (see e.g. \cite{MR1945684,MR2167641,MR2206341,MR2399292,MR2682263,MR2795050,MR2797949,MR2827968,MR3069372} for results on related models of patterned random matrices, including circulant, Hankel and Toeplitz matrices). Indeed, the usual Wigner or sample covariance type matrices are built with a number of independent random variables of the same order as the number of entries of the matrix, whereas if $X=(X(g))_{g\in G}$ is a family of independent random variables, then $P_X$ is a $|G|\times |G|$ matrix built with only $|G|$ independent random variables.

The purpose of this article is to extend the analysis of spectral properties of $P_X$ presented in \cite{MR3069372} to the case of non-Abelian groups by relating the spectral distribution of $P_X$ to the Plancherel measure of $G$. Recall that the spectral measure of an $N \times N$ matrix $A$ is the Borel probability measure on $\C$ given by the formula
\begin{displaymath}
  L_A = \frac{1}{N}\sum_{i=1}^N \delta_{\lambda_i},
\end{displaymath}
where $\delta_x$ stands for the Dirac mass at $x$ and $\lambda_1,\ldots,\lambda_N$ are the eigenvalues of $A$, counted with their algebraic multiplicities. If $A$ is a random matrix, then $L_A$ becomes a random probability measure and its asymptotic behavior as $N\to \infty$ for various models of random matrices is one of the main topics of Random Matrix Theory. Usually one considers two types of spectral measures -- $L_A$ and $L_{\sqrt{AA^\ast}}$, which describe the collective (global) behavior of respectively eigenvalues and singular values of $A$. Note that $L_{\sqrt{AA^\ast}}$ is supported on $[0,\infty)$, whereas in general $L_A$ is a measure on $\C$.

If $A = A_N$ is an $N \times N$ random matrix with i.i.d. entries of mean zero and variance one, the classical theorem by Marchenko and Pastur \cite{MR0208649} provides almost sure weak convergence of $L_{\frac{1}{\sqrt{N}} \sqrt{AA^\ast}}$ to the deterministic probability measure with density
\begin{displaymath}
  \rho(x) = \frac{1}{\pi}\sqrt{4 - x^2}\ind{[0,2]}(x).
\end{displaymath}
The asymptotic behavior of $L_{\frac{1}{\sqrt{N}}A}$ is on the other hand given by the uniform measure on the unit disc in the complex plane as shown in full generality by Tao and Vu  \cite{MR2722794} (closing a long line of research initiated by Ginibre and Mehta \cite{MR0173726, MR1083764} and subsequently developed by Girko \cite{MR773436}, Edelman \cite{MR1437734}, Bai \cite{MR1428519},  Pan-Zhou \cite{MR2575411}, G\"otze-Tikhomirov \cite{MR2663633}). We refer to the books \cite{MR2760897,MR2567175,MR1083764, MR2808038,MR2906465} for detailed exposition of this and other aspects of Random Matrix Theory.

In the case of Abelian circulant matrices, it was shown by Meckes \cite{MR2797949} that if $G_N = \Z_N$ and $X^N=(X_i^N)_{i=0}^{N-1}$ is a sequence of independent mean zero, random variables with $\E |X_i^N|^2 = 1$, satisfying a Lindeberg type condition (in particular if they are all copies of a single square integrable random variable), then the spectral measure of
\begin{displaymath}
\frac{1}{\sqrt{N}}P_{X^N} = \frac{1}{\sqrt{N}}[X_{(i-j)\,\textrm{mod}\, N}^N]_{i,j=0}^{N-1}
\end{displaymath}
converges weakly in probability to the standard Gaussian measure in the complex plane. Since in this case $P_X$ is a normal matrix, this implies that the spectral measure of $\frac{1}{\sqrt{N}}\sqrt{P_X P_X^\ast}$ converges to the measure with density $f(x) = 2xe^{-x^2}$ on $[0,\infty)$ (see also \cite{MR2682263} for related results concerning singular values of rectangular patterned matrices). Extending this result to more general Abelian setting, Meckes showed \cite{MR3069372} that if $G_N$ is a sequence of Abelian groups with $|G_N|\to \infty$, then the same limiting behavior holds for $\frac{1}{\sqrt{|G_N|}} \sqrt{P_{X^N} P_{X^N}^\ast}$ and $\frac{1}{\sqrt{|G_N|}} P_{X^N}$ under an analogous Lindeberg condition on $(X^N_g)_{g \in G_N}$ and an additional assumption that $\E (X_g^N)^2 = 0$. We remark that the results by Meckes cover also the case when $\E (X_g^N)^2$ is a nonzero real constant, however we will not be concerned with this case in the article so we postpone the full formulation of the result till the final section, where we discuss possible lines of future research.

An important fact that distinguishes structured Abelian circulant matrices from the case of matrices with i.i.d. entries is their normality, already mentioned above.  For Abelian groups we have $P_X P_X^\ast = P_X^\ast P_X$, whereas large matrices $A$ with i.i.d. entries are with high probability non-normal, which makes the analysis of their eigenvalues (i.e. the measure $L_A$) much more complicated than for singular values (i.e. for the measure $L_{\sqrt{AA^\ast}}$).

For non-Abelian $G$ one again loses the normality of $P_X$, and as a consequence the analysis of eigenvalues and singular values of $P_X$ are two different problems. As we will see, the behavior of spectral measures in the general case turns out to be closely related to that of classical random matrix ensembles. The limiting spectral distribution is however more complicated and therefore before introducing our main results we need to recall te definition of the Plancherel measure of a group $G$. Denote by $\widehat{G}$ the set of (equivalence classes) or irreducible unitary representations of $G$ (we refer to the first chapters of \cite{MR964069, MR0450380} for the basic facts from representation theory, which we are going to use). The Plancherel measure on $\widehat{G}$ is defined as
\begin{displaymath}
  \mu_G(\Lambda) = \frac{(\dim \Lambda)^2}{|G|}.
\end{displaymath}
It follows from basic principles of representation theory that $\mu_G$ is a probability measure. In what follows we are going to need the image $\widetilde{\mu}_G$ of $\mu_G$ under the map $T \colon \widehat{G} \to \Z_+$ := \{1,2,\ldots\}, given by $T(\Lambda) = \dim \Lambda$. Thus $\widetilde{\mu}_G$ is a probability measure on $\Z_+$. It will be however more convenient to view $\widetilde{\mu}$ as a probability measure on the one point compactification $\bar{\Z}_+ = \Z_+\cup \{\infty\}$. Explicitly we have
\begin{align}\label{eq:mu-tilde}
  \widetilde{\mu}_G(\{n\}) = \frac{n^2}{|G|}|\{ \Lambda \in \widehat{G}\colon \dim \Lambda = n\}|.
\end{align}

The last ingredient we will need to formulate our results is the definition of the measures appearing in the analysis of classical random matrix ensembles.

\begin{defi}
  For a positive integer $n$ define the measures $\rho_{n}$, $\theta_n$ in the following way.
  \begin{itemize}
    \item Let $\rho_n$ be the measure on $[0,\infty)$ with density

        \begin{displaymath}
          \frac{d\rho_n(x)}{dx} = 2xe^{-nx^2}\sum_{l=0}^{n-1} (\mathcal{L}_l(nx^2))^2,
        \end{displaymath}
        where
        \begin{displaymath}
     \mathcal{L}_l(x) = \frac{e^x}{l!}\frac{d^l}{dx^l}(e^{-x}x^l) = \sum_{k=0}^l \binom{l}{k}\frac{(-1)^k}{k!}x^k
    \end{displaymath}
    is the $l$-th Laguerre polynomial.
    \item Let $\theta_n$ be the measure on $\C$ with density
        \begin{displaymath}
          \frac{d\theta_n(z)}{dz} = \frac{1}{\pi}e^{-n|z|^2}\sum_{l=0}^{n-1}\frac{n^l|z|^{2l}}{l!}.
        \end{displaymath}
  \end{itemize}
    Define also $\rho_\infty$ as the measure with density $\frac{1}{\pi}\sqrt{4 - x^2}\ind{[0,2]}(x)$ (the quarter circle law) and $\theta_\infty$ as the uniform measure on the unit disc in the complex plane (the circular law).

\end{defi}

\begin{remark}\label{rem:rho-and-theta-explained}
  The measure $\rho_n$ and $\theta_n$ are the mean spectral distributions of respectively singular and eigenvalues of the $n\times n$ complex Ginibre ensemble, i.e. of the random matrix $\Gamma_n = \frac{1}{\sqrt{n}}A_n$, where $A_n$ is an $n\times n$ matrix whose entries are i.i.d. standard complex Gaussian variables (i.e. their real and imaginary parts are independent and have Gaussian distribution with mean zero and variance 1/2), see \cite[Theorem 7.5.1]{MR2808038} and \cite[Theorem 3.3]{MR2908617}. The measures $\rho_\infty$ and $\theta_\infty$ are just the weak limits of $\rho_n$ and $\theta_n$ respectively (as asserted by the famous Marchenko-Pastur  \cite{MR0208649} and circular law theorems \cite{MR2908617}, already mentioned above).
\end{remark}

In order to formulate our results let us recall the notion of weak convergence in probability.
\begin{defi}
Let $\nu_N$ be a sequence of random Borel probability measures on $\C$ and let $\nu$ be a deterministic Borel probability measure on $\C$. We will say that the sequence $\nu_N$ converges to $\nu$ weakly in probability if for all $\varepsilon > 0$,
\begin{displaymath}
\lim_{N\to \infty} \p(d(\nu_N,\nu) > \varepsilon) = 0,
\end{displaymath}
where $d$ is a distance on the set of probability measures metrizing the weak convergence.
\end{defi}

\begin{remark}
It is easy to see that the choice of the distance $d$ in the above definition is irrelevant.
\end{remark}

Our main result concerning the distribution of singular values of convolutions is the following theorem.

\begin{theorem}\label{thm:singular-values}
  Let $G_N$ be as sequence of finite groups with $|G_N|\to \infty$. Assume that the sequence of measures $\widetilde{\mu}_{G_N}$ converges weakly to a probability measure $\mu$ on $\overline{\Z}_+$. Let $\xi$ be a complex random variable such that $\E \xi = \E \xi^2 = 0$, $\E|\xi|^2 = 1$.  For each $N$ let $X^N = (X_g)_{g\in G_N}$  be a family of independent copies of $\xi$. Denote finally by $L_n^s$ the (random) spectral measure of $\frac{1}{\sqrt{|G_N|}} \sqrt{P_{X^N}P_{X^N}^\ast}$. Then the sequence of random measures $L_n^s$ converges weakly in probability to the deterministic measure $L^s_\infty$ on $[0,\infty)$ with density
  \begin{align}\label{eq:singular-values-limit-formula}
    \frac{dL^s_\infty(x)}{dx} = \sum_{n \in \overline{\Z}_+} \mu(n) \frac{d\rho_n(x)}{dx}.
  \end{align}
  In particular if $\widetilde{\mu}_{G_N}$ converges weakly to Dirac's mass at $\infty$ then $L^s_\infty$ is the quarter circle law.
\end{theorem}

\begin{remark}
  The formula \eqref{eq:singular-values-limit-formula} shows that $L_\infty^s$ is just a mixture of the measures $\rho_n$ with the directing measure equal to $\mu$. In the Abelian case all the measures $\mu_n$ are concentrated on $\{1\}$ thus $L^s_\infty$ has density $2x e^{-x^2}$ and we recover the behaviour described by Meckes \cite{MR3069372}.
\end{remark}

\begin{remark}
Instead of assuming identical distribution, one can also consider Lindeberg type conditions and obtain similar results. To avoid further complications of statements, we will however stick to the setting described in Theorem \ref{thm:singular-values}.
\end{remark}

\begin{remark}
  Note that due to compactness of $\overline{\Z}_+$, for every sequence of groups $G_N$, the sequence of measures $\widetilde{\mu}_N$ is precompact in the topology of weak convergence. If $|G_N|\to \infty$, the above theorem describes the behaviour of the spectral measure along each subsequence $N_n$ such that $\mu_{N_n}$ converges.

\end{remark}

In fact Theorem \ref{thm:singular-values} easily implies the following

\begin{theorem}\label{thm:singular-values-uniform version}
Let $d$ be any distance metrizing the weak convergence of probability measures on $\R$ and let $\xi$ be a complex random variable such that $\E \xi = \E \xi^2 = 0$, $\E|\xi|^2 = 1$.  Then for every $\varepsilon, \delta >0$ there exists $N> 0$ such that for any group $G$ of cardinality at least $N$ the following property holds. If $X = (X_g)_{g\in G}$  is a family of independent copies of $\xi$, $L_{G}^s$ is the (random) spectral measure of $\frac{1}{\sqrt{|G|}} \sqrt{P_{X}P_{X}^\ast}$ and $\mathcal{L}_G$ is the measure defined by formula
  \begin{align}
    \frac{d\mathcal{L}_G(x)}{dx} = \sum_{n \in \overline{\Z}_+} \widetilde{\mu}_G (n) \frac{d\rho_n(x)}{dx},
  \end{align}
  then
  \begin{displaymath}
  \p(d(L_G^s,\mathcal{L}_G) > \varepsilon) < \delta.
\end{displaymath}
\end{theorem}

\medskip

Let us now pass to eigenvalues. As already mentioned above, the problem of the limiting distribution of eigenvalues is usually more subtle than that of singular values. For this reason, in our main result we need to restrict to Gaussian variables.

\begin{theorem}\label{thm:eigenvalues}
  Assume that $\xi$ is a standard complex Gaussian variable and consider the sequence of groups $G_N$ and associated convolution operators $P_{X^N}$ as in Theorem \ref{thm:singular-values}. Then the empirical spectral measure $L^e_N$ of the matrix $\frac{1}{\sqrt{|G_N|}}P_{X^N}$ converges weakly in probability to the deterministic probability measure $L^e_\infty$ on $\C$, with density
  \begin{align}\label{eq:eigenvalues-limit-formula}
    \frac{dL^e_\infty(z)}{dz} = \sum_{n \in \overline{\Z}_+} \mu(n) \frac{d\theta_n(z)}{dz}.
  \end{align}

  In particular if $\widetilde{\mu}_{G_N}$ converges weakly to Dirac's mass at $\infty$ then $L_e^\infty$ is the uniform distribution on the unit disc in the complex plane.
\end{theorem}

It turns out that under additional assumption that the limiting measure $\mu$ is the Dirac mass at infinity, one can easily analyze polynomials in several independent random convolution operators. In particular, if we assume that the underlying random variable $\xi$ has all moments, such independent random matrices are asymptotically free and converge in moments to circular elements. Since the formulation of the corresponding result is rather technical, we postpone it to Section \ref{sec:freeness} (see Proposition \ref{prop:freeness} and Theorem \ref{thm:freeness-nongaussian}).

\medskip
Another aspect one can investigate is the Central Limit Theorem for linear eigenvalue statistics. We provide such a result in Theorem \ref{thm:CLT-Gaussian}.

\medskip

The organization of the article is as follows. First, in Section \ref{sec:examples} we present a few examples to illustrate our main theorems with concrete well-studied sequences of groups. Next, in Section \ref{sec:Gaussian} we discuss the Gaussian case, in particular we prove Theorem \ref{thm:eigenvalues}. Section \ref{sec:universality} is devoted to the proof of Theorem \ref{thm:singular-values}. In Section \ref{sec:freeness} we discuss asymptotic freeness of collections of independent random projection operators while Section \ref{sec:CLT} is devoted to Central Limit Theorems for the Gaussian case.  We conclude with a list of problems for future research (Section \ref{sec:questions}). Auxiliary lemmas together with their proofs are postponed to the Appendix.

\medskip

\paragraph{\bf Acknowledgement} The Author would like to thank Piotr {\'S}niady for his patience with answering numerous questions concerning representation theory of finite groups.

\section{Examples}\label{sec:examples}

In this section we will illustrate our main results with examples related to classical sequences of groups.
\medskip

\begin{enumerate}[(I)]
\item \label{en:Abelian}If the groups $G_N$ are Abelian, then all irreducible representations are of dimension one. As a  consequence we get $L^e_\infty = \theta_1$ -- the standard Gaussian measure on $\C$ and $L^s_\infty = \rho_1$ -- the measure with density $2xe^{-x^2}$ on $[0,\infty)$. We thus recover the result by Meckes concerning Abelian circulants \cite{MR3069372}.
    \medskip
\item Let $G_N$ be the $N$-th dihedral group, i.e. the isometry group of the regular $N$-sided polygon. It is a group of order $2N$ and, as is well known for $N$ odd it has $2$ one-dimensional representations and $(N-1)/2$ representation of dimension two, while for $N$ even it has $4$ one-dimensional and $(N-2)/2$ two-dimensional representations. Thus $\widetilde{\mu}_{G_N}$ converges to the Dirac's mass at two and we have $L^e_\infty = \theta_2$, $L^s_\infty = \rho_2$.
    \medskip
\item If $G_N = S_N$ -- the symmetric group of permutations of an $N$-element set, then $\widetilde{\mu}_{G_N}$ converges to Dirac's mass at $\infty$ (see \cite{MR783703} for much more precise results concerning the asymptotics of this measure). Thus the limiting eigenvalue and singular value distributions are in this case the circular and quarter-circular laws $\theta_\infty, \rho_\infty$.
    \medskip
\item Let $F$ be a finite field of cardinality $q$ and $G_q = GL_2(F)$ be the linear group of $2 \times 2$ invertible matrices over $F$. Then $G_q$ has $q-1$ representations of dimension one, $q(q-1)/2$ representation of dimension $q-1$, , $q-1$ representations of dimension $q$, $(q-1)(q-2)/2$ representations of dimension $q+1$. Thus for any admissible sequence $q_N \to \infty$, the measure $\widetilde{\mu}_{G_{q_N}}$ converges to Dirac's mass at infinity and so the limiting eigenvalue and singular value distributions for random convolution operators on $G_{q_N}$ are again the circular and quarter-circular distributions.
    \medskip
\item Fix a finite group $G$ and a sequence of Abelian groups $H_N$ with cardinalities tending to infinity and let $G_N = G \times H_N$. Then using the fact that irreducible representations of direct products are given by tensor products of irreducible representations of the factors, we obtain $\widetilde{\mu}_{G_N} = \widetilde{\mu}_G$. Thus in this case $\mu = \widetilde{\mu}_G$ and the limiting eigenvalue and singular value distributions are mixtures of the measures $\theta_n$ and $\rho_n$ respectively, driven by the measure $\widetilde{\mu}_G$.
    To give a concrete example, let us consider $G_N = S_3\times \Z_N = D_3\times Z_N$. Then $L^e_\infty = \frac{1}{3}\theta_1 + \frac{2}{3}\theta_2$, $L^s_\infty = \frac{1}{3}\rho_1 + \frac{2}{3}\rho_2$.

\medskip
\item Fix a finite group $G$ and let $G_N = G^{\times N}$ be the $N$-fold direct product of $G$. Then, for $G$ -- Abelian, by item \eqref{en:Abelian}, we clearly have $L^e_\infty = \theta_1$, $L^s_\infty = \rho_1$, whereas for $G$ -- non-Abelian we have $L^e_\infty = \theta_\infty$ -- the circular law and $L^s_\infty = \rho_\infty$ -- the quarter-circular law. Thus the asymptotic behavior of random convolution operators on $G^{\times N}$ `encodes' the commutativity of $G$.
\medskip
\end{enumerate}

A natural question arising in view of our results and the above examples is what measures can be obtained as limits of eigenvalue or singular value distributions of random convolution operators, i.e. which measures $\mu$ on $\overline{\Z}_+$ are weak limits of projected
Plancherel measures $\widetilde{\mu}_{G_N}$ for some sequence of groups $G_N$. By using again the form of irreducible representations of direct products one can easily see that the set of all limits is closed under multiplicative convolutions, i.e. if $\mu$ and $\nu$ are such limits then so is the measure $\lambda$ given by the formula $\lambda({n}) = \sum_{k|n} \mu(\{k\})\nu(\{n/k\})$. Moreover, as noted above it contains all projected Plancherel measures of finite groups and the Dirac mass at two. The characterization of all possible weak limits is beyond the scope of mostly probabilistic techniques which lie behind our main results and seems to require a more detailed analysis on the grounds of representation theory.

\section{The Gaussian case}\label{sec:Gaussian}

In this section we will consider the case of a convolution operator $P_X$ on a finite group $G$, where $X = (X_g)_{g\in G}$ is a family of independent standard complex Gaussian variables, i.e. for each $g \in G$, the random variables $\Re X_g$ and $\Im X_g$ are independent real Gaussian variables with mean zero and variance $1/2$.

Let us start by recalling some basic facts from noncommutative harmonic analysis on finite groups, which will allow us to give an explicit description of the spectrum of random convolution operators in the Gaussian case.

Denote by $L_2(G)$ the space of complex valued sequences on $G$, with the inner product $\langle x, y \rangle = \frac{1}{|G|}\sum_{g \in G} x_g \bar{y}_g$ and let $\|\cdot\|$ be the corresponding norm.
Recall from representation theory that there is a natural isometry between the space $L_2(G)$ and the direct sum $\oplus_{\Lambda \in \widehat{G}} M(\dim \Lambda)$, where $M(n)$ is the space of complex $n \times n$  matrices with the inner product given by $\langle A, B\rangle = \frac{1}{n}\tr A B^\ast$. More precisely, the functions $\{ g \mapsto \Lambda_{ij}(g) \colon \Lambda \in \widehat{G}, i,j \in \{1,\ldots,\dim \Lambda\}\}$ form an orthogonal basis of $L_2(G)$ with $\|\Lambda_{ij}\|_2 = \frac{1}{\sqrt{\dim \Lambda}}$ (see e.g. \cite[Chapter 2C, Corollary 3]{MR964069}).

Recall also the definition of the Fourier transform of a sequence $x \in L_2(G)$, which we will denote by $\widehat{x}$ and which is a function defined on $\widehat{G}$ with the formula
\begin{displaymath}
\widehat{x}(\Lambda) = \sum_{g \in G} x_g \Lambda(g).
\end{displaymath}

Thus $\widehat{x}(\Lambda)$ is a square matrix of size $\dim \Lambda$. The Plancherel formula asserts that for all functions $x, y\colon G\to \C$,
\begin{align}\label{eq:Plancherel}
\sum_{g \in G} y_{g^{-1}}x_g = \frac{1}{|G|}\sum_{\Lambda \in \widehat{G}} \dim(\Lambda) \tr \widehat{y}(\Lambda)\widehat{x}(\Lambda).
\end{align}

In particular we have the Fourier inversion formula
\begin{displaymath}
x_g = \frac{1}{|G|} \sum_{\Lambda \in \widehat{G}} (\dim \Lambda)\tr (\Lambda(g^{-1}) \widehat{x}(\Lambda))
\end{displaymath}
and so the map $x \mapsto \widehat{x}$ can be regarded as a linear isomorphism of $L_2(G)$ and $\oplus_{\Lambda \in \widehat{G}} M(\dim \Lambda)$.

We will make use of the fundamental relation between the Fourier transform and convolutions, namely for any two sequences $x,y \in L_2(G)$ and any $\Lambda \in \widehat{G}$,
\begin{displaymath}
\widehat{x\ast y}(\Lambda) = \widehat{x}(\Lambda) \widehat{y}(\Lambda),
\end{displaymath}
where on the right hand side we have multiplication of matrices. It follows that if for an $n\times n$ matrix $Q$, the linear map $M_Q \colon M(d) \to M(d)$ is given by $M_Q(A) = QA$, then the spectrum of $P_x$ (treated as a multiset) is the same as the spectrum of the linear operator
$\mathcal{M}_x \colon \oplus_{\Lambda \in \widehat{G}} M(\dim \Lambda) \to \oplus_{\Lambda \in \widehat{G}} M(\dim \Lambda)$ defined as $\mathcal{M}_x = \oplus_{\Lambda \in \widehat{G}} M_{\widehat{x}(\Lambda)}$, i.e.
$\mathcal{M}_x (\oplus_{\Lambda \in  \widehat{G}} A_\Lambda) = \oplus_{\Lambda \in \widehat{G}} (\widehat{x}(\Lambda) A_\Lambda)$.

\medskip

One has the following simple proposition.
\begin{prop}
Let $Q$ be an $n\times n$ matrix and let $\rho$ be its eigenvalue of algebraic multiplicity $m$. Then $\rho$ is an eigenvalue of $M_Q$ of algebraic multiplicity $nm$.
\end{prop}

\begin{proof}
If $\rho$ has multiplicity one and $v$ is an eigenvector corresponding to $\rho$, then the matrix $A_i$ with the $i$-th column equal $v$ and the remaining columns equal zero is an eigenvector of $M_Q$ with eigenvalue $\rho$. In the special case, when all eigenvalues of $Q$ have multiplicity one, this ends the proof, since the $n$ just described eigenspaces of dimension $n$ each span the whole space $M(n)$. The general case follows by continuous dependence of the spectrum (treated as multiset with algebraic multiplicities) on the matrix  (see e.g. \cite[Appendix D]{MR832183}) and the fact that matrices with all eigenvalues of multiplicity one are dense in $M(n)$.
Alternately one can identify $M_Q$ with $Q\otimes \id_n$, where $\id_n$ is the $n\times n$ identity matrix and use the relation between the spectrum of tensor products of matrices and spectra of their factors.
\end{proof}

Applying this result to $P_x$ we obtain the following
\begin{corollary}\label{cor:deterministic-spectrum-explained}
The spectral measure of $P_x$ (resp. $\sqrt{P_xP_x^\ast}$) is the mixture of the spectral measures of $\widehat{x}(\Lambda)$, $\Lambda \in \widehat{G}$ (resp. $\sqrt{\widehat{x}(\Lambda)\widehat{x}^\ast(\Lambda)}$, $\Lambda \in \widehat{G}$) with the driving measure equal to the Plancherel measure on $\widehat{G}$.
\end{corollary}

\begin{proof}

We have
\begin{align*}
L_{P_x} & = \frac{1}{|G|} \sum_{\lambda \in \spec(P_x)} \delta_{\lambda} \\
& = \sum_{\Lambda \in \widehat{G}} \frac{1}{|G|}\sum_{\lambda \in \spec(M_{\widehat{x}(\Lambda)})} \delta_{\lambda} \\
& = \sum_{\Lambda \in \widehat{G}} \frac{\dim(\Lambda)^2}{|G|} \frac{1}{\dim(\Lambda)} \sum_{\lambda \in \spec(\widehat{x}(\Lambda))}\delta_\lambda\\
& =  \sum_{\Lambda \in \widehat{G}} \frac{\dim(\Lambda)^2}{|G|} L_{\widehat{x}(\Lambda)},
\end{align*}
which proves the assertion for $P_x$. The statement for $\sqrt{P_{x} P_{x}^\ast}$ follows from two elementary observations. First, $P_x^\ast$ is equal to $P_y$, where for $g \in G$, $y(g) = \overline{x(g^{-1})}$ and thus by associativity of the convolution we get $P_x P_x^\ast = P_{x\ast y}$, which by the already proved part of the corollary yields
\begin{displaymath}
L_{P_x P_x^\ast} = L_{P_{x\ast y}} = \sum_{\Lambda \in \widehat{G}} \frac{\dim(\Lambda)^2}{|G|} L_{\widehat{x\ast y}(\Lambda)}.
\end{displaymath}
Next, we have $\widehat{x\ast y}(\Lambda) = \widehat{x}(\Lambda)\widehat{y}(\Lambda)$ and
\begin{displaymath}
\widehat{y}(\Lambda) = \sum_{g \in G} y(g) \Lambda(g) = \sum_{g \in G} \overline{x(g)} \Lambda(g^{-1}) = \Big(\sum_{g\in G} x(g) \Lambda(g)\Big)^\ast = \widehat{x}(\Lambda)^\ast,
\end{displaymath}
where in the last line we used the fact that $\Lambda(g^{-1}) = \Lambda(g)^{-1} = \Lambda(g)^\ast$ as the representations we consider are unitary.

Thus $\widehat{x\ast y}(\Lambda) = \widehat{x}(\Lambda)\widehat{x}(\Lambda)^\ast$ and we get
\begin{displaymath}
L_{P_x P_x^\ast} = L_{P_{x\ast y}} = \sum_{\Lambda \in \widehat{G}} \frac{\dim(\Lambda)^2}{|G|} L_{\widehat{x}(\Lambda)\widehat{x}(\Lambda)^\ast}.
\end{displaymath}
Passing from $P_x P_x^\ast$ to $\sqrt{P_x P_x^\ast}$ is straightforward as for a nonnegative definite matrix $A$ the eigenvalues of $\sqrt{A}$ are the square roots of the eigenvalues of $A$.
\end{proof}

We can now easily obtain the description of the distribution of the spectral measure of a random convolution operator given by a sequence of standard complex Gaussian variables in terms of independent Ginibre ensembles.

\begin{prop}\label{prop:Gaussian-distribution}Let $G$ be a finite group and let $X = (X_g)_{g\in G}$ be a family of i.i.d. standard complex Gaussian random variables. Let also $(\Gamma_\Lambda)_{\Lambda \in \widehat{G}}$ be a family of independent Ginibre random matrices (i.e. for each $\Lambda$, $\sqrt{\dim(\Lambda)}\Gamma_\Lambda$ is a matrix with i.i.d. standard complex Gaussian entries). Then
\begin{itemize}
\item[(i)] the spectral measure $L_{\frac{1}{\sqrt{|G|}}P_X}$ has the same distribution as the random measure
\begin{displaymath}
\sum_{\Lambda \in \widehat{G}} \mu_G(\Lambda) L_{\Gamma_\Lambda},
\end{displaymath}
and
\item[(ii)] the spectral measure $L_{\frac{1}{\sqrt{|G|}}\sqrt{P_XPX^\ast}}$ has the same distribution as the random measure
\begin{displaymath}
\sum_{\Lambda \in \widehat{G}} \mu_G(\Lambda) L_{\sqrt{\Gamma_\Lambda \Gamma_\Lambda^\ast}}.
\end{displaymath}
\end{itemize}
\end{prop}

From the above proposition and Remark \ref{rem:rho-and-theta-explained} we immediately obtain

\begin{corollary}\label{cor:Gaussian-expected-measure}
Under the assumptions of Proposition \ref{prop:Gaussian-distribution}, the average spectral measure of $\frac{1}{\sqrt{|G|}}P_X$ satisfies
\begin{displaymath}
\bar{L}_{\frac{1}{\sqrt{|G|}}P_X} := \E L_{\frac{1}{\sqrt{|G|}} P_X} = \sum_{n \in \Z_+} \widetilde{\mu}_G(n)\theta_n.
\end{displaymath}
Similarly, the average spectral measure of $\frac{1}{\sqrt{|G|}}\sqrt{P_XP_X^\ast}$ equals
\begin{displaymath}
\bar{L}_{\frac{1}{\sqrt{|G|}}\sqrt{P_XP_X^\ast}} := \E L_{\frac{1}{\sqrt{|G|}}\sqrt{P_XP_X^\ast}} = \sum_{n \in \Z_+} \widetilde{\mu}_G(n)\rho_n.
\end{displaymath}
\end{corollary}

\begin{proof}[Proof of Proposition \ref{prop:Gaussian-distribution}]
In view of Corollary \ref{cor:deterministic-spectrum-explained} it is enough to show that the entries of the matrices $\widehat{X}(\Lambda)$, $\Lambda \in \widehat{G}$ are jointly independent and that for all $\Lambda$ and $i,j =1,\ldots,\dim(\Lambda)$, the random variable $\sqrt{\frac{\dim(\Lambda)}{|G|}}\widehat{X}(\Lambda)_{ij}$ has a standard complex Gaussian distribution. Note first that
the random variables $\Re \widehat{X}(\Lambda)_{ij}, \Im \widehat{X}(\Lambda)_{ij}, \Lambda \in \widehat{G}, i,j \in \{1,\ldots,\dim(\Lambda)\}$ are jointly Gaussian as linear transformations of the Gaussian vector $(\Re X_g, \Im X_g\colon g \in G)$. Thus it remains to calculate the covariances.

Recall that for any $g \in G$ we have $\E X_g = \E X_g^2 = 0$ and $\E |X_g|^2 = 1$. We will now use the following lemma, which is proved below.

\begin{lemma}\label{le:covariances}
Let $X = (X_g)_{g \in G}$ be independent random variables such that for each $g \in G$, $\E X_g = \E X_g^2 = 0$, $\E |X_g|^2 = 1$. Then
the real random variables $\Re \widehat{X}(\Lambda)_{ij}, \Im \widehat{X}(\Lambda)_{ij}$, $\Lambda \in \widehat{G}$, $i,j \in \{1,\ldots,\dim(\Lambda)\}$
are pairwise uncorrelated. Moreover for all such $\Lambda, i, j $,
\begin{displaymath}
\E \widehat{X}(\Lambda)_{ij} = 0, \; \E (\Re \widehat{X}(\Lambda)_{ij})^2 = \E (\Im \widehat{X}(\Lambda)_{ij})^2 = \frac{|G|}{2\dim(\Lambda)}.
\end{displaymath}
\end{lemma}
Thus  the random variables $\sqrt{\dim(\Lambda)/|G|}\Re \widehat{X}(\Lambda)_{ij}, \sqrt{\dim(\Lambda)/|G|} \Im \widehat{X}(\Lambda), \Lambda \in \widehat{G}$, $i,j \le \dim(\Lambda)$ indeed form a family of i.i.d. Gaussian variables of mean zero and variance 1/2.
\end{proof}

\begin{proof}[Proof of Lemma \ref{le:covariances}]

Obviously $\E \widehat{X}(\Lambda)_{ij} = 0$. Moreover, using independence of $X_g$'s, for any $\Lambda,\Delta \in \widehat{G}$ $i,j \le \dim(\Lambda)$ and $k,l \le \dim(\Delta)$ we get
\begin{align*}
\E \widehat{X}(\Lambda)_{ij} \overline{\widehat{X}(\Delta)_{kl}} &= \E \Big(\sum_{g \in G} X_g \Lambda(g)_{ij}\Big)\overline{\Big(\sum_{g \in G} X_g \Delta(g)_{kl}\Big)}\\
&  = \sum_{g \in G} \Lambda(g)_{ij} \overline{\Delta(g)_{kl}}.
\end{align*}
By basic properties of representations (see e.g. \cite[Chapter 2B, Corollary 2,3]{MR964069}), the latter sum equals $\frac{|G|}{\dim(\Lambda)}$ if $\Lambda = \Delta$, $i=k$, $j= l$ and vanishes otherwise.
We also have $\E X_g^2 = 0$, which gives
\begin{displaymath}
\E \widehat{X}(\Lambda)_{ij} \widehat{X}(\Delta)_{kl} = 0
\end{displaymath}
for all $\Lambda,\Delta \in \widehat{G}$, $i,j \le \dim(\Lambda)$, $k,l\le \dim(\Delta)$.

It is easy to see that for real random variables $a,b$, the conditions $\E|a + b\sqrt{-1}|^2 = 1$, $\E (a + b\sqrt{-1})^2 = 0$ hold if and only if $\E a^2 = \E b^2 = 1/2$ and $\E ab = 0$, which together with the above equalities shows that for any $\Lambda$ and $i,j \le \dim(\Lambda)$, $\Re \widehat{X}(\Lambda)_{ij}, \Im \widehat{X}(\Lambda)_{ij}$ are indeed uncorrelated with mean zero and variance $\frac{|G|}{2\dim(\Lambda)}$. Similarly, if $c,d$ is another pair of real random variables, then $\E(a+b\sqrt{-1})(c+d\sqrt{-1}) = \E (a+b\sqrt{-1})(c-d\sqrt{-1}) = 0$ if and only if $\E ac = \E ad = \E bc = \E bd = 0$, which shows that the remaining covariances also vanish, thus ending the proof.
\end{proof}

We are now ready to prove Theorem \ref{thm:eigenvalues}.

\begin{proof}[Proof of Theorem \ref{thm:eigenvalues}]
  To simplify the notation, in what follows we are going to drop the superscript $e$ and write simply $L_N, L_\infty$ instead of $L^e_N$, $L^e_\infty$.
  It is well known (see Lemma \ref{le:convergence-functions} in the Appendix) that $L_N$ converges to $L_\infty$ weakly in probability if and only if for every bounded continuous function $f \colon \R\to \R$, the sequence of complex random variables $\int_{\C} f dL_N$ converges in probability to $\int_\C fdL_\infty$. Consider thus a bounded continuous function $f$.

  Note that by the Ginibre-Mehta circular law, as $n \to \infty$, the measures $\theta_n$
  converge weakly to $\theta_\infty$ and also if $\Gamma_n$ is an $n\times n$ random matrix, such that the entries of $\sqrt{n}\Gamma_n$ are i.i.d. standard complex Gaussian variables, then $L_{\Gamma_n}$ converges weakly in probability to $\theta_\infty$. In particular, as $n\to \infty$ we have $\int_{\C} f dL_{\Gamma_n} \to \int_{\C} fd\theta_\infty$ in probability and thanks to boundedness of $f$ also in $L_2$.

  Fix now $\varepsilon > 0$ and let $n_0$ be such that

  \begin{itemize}
      \item for $n > n_0$,
        \begin{align}\label{eq:L_2-distance}
            \Big\|\int_{\C} f dL_{\Gamma_n} - \int_\C f d\theta_\infty\Big\|_2 < \varepsilon,
        \end{align}
    and
    \item for $n > n_0$,
        \begin{align}\label{eq:error-at-infinity}
            \Big|\int_{\C} f d\theta_n - \int_\C f d\theta_\infty\Big| < \varepsilon.
        \end{align}
  \end{itemize}

Recall the notation of Proposition \ref{prop:Gaussian-distribution}.
We will now estimate the $L_2$ error
\begin{align*}
  \mathcal{E}_N := \Big\|\int_\C f dL_N - \int_\C fdL_\infty \Big\|_2& =  \Big\|\sum_{\Lambda \in \widehat{G}_N} \mu_{G_N}(\Lambda) \int_\C f d L_{\Gamma_\Lambda} - \sum_{n\in\overline{\Z}_+}\mu(n) \int_\C fd\theta_n\Big\|_2.
\end{align*}

Recall that $\widetilde{\mu}_{G_N}(n) = \sum_{\stackrel{\Lambda \in \widehat{G}_N}{\dim \Lambda = n}} \mu_{G_N}(\Lambda)$, therefore by the triangle inequality we get
\begin{align}\label{eq:5-term-decomposition}
\mathcal{E}_N \le \sum_{k=1}^5 \mathcal{E}_{k,N},
\end{align}
where
\begin{align*}
  \mathcal{E}_{1,N}& = \Big\|\sum_{\stackrel{\Lambda \in \widehat{G}_N}{\dim \Lambda \le n_0}} \mu_{G_N}(\Lambda) \Big(\int_\C f d L_{\Gamma_\Lambda} - \int_\C fd\theta_{\dim \Lambda}\Big)\Big\|_2,\\
  \mathcal{E}_{2,N} & =  \Big|\sum_{n=1}^{n_0}(\widetilde{\mu}_{G_N}(n) - \mu(n))\int_\C fd\theta_n\Big| ,\\
  \mathcal{E}_{3,N} &= \Big\|\sum_{\stackrel{\Lambda \in \widehat{G}}{\dim(\Lambda) > n_0}}  \mu_{G_N}(\Lambda)  \Big(\int_\C f dL_{\Gamma_\Lambda} -  \int_\C f d\theta_\infty\Big)\Big\|_2,\\
  \mathcal{E}_{4,N} &= \Big|\Big(\widetilde{\mu}_{G_N}([n_0+1,\infty]) - \mu([n_0+1,\infty])\Big) \int_\C f d\theta_\infty \Big|,\\
  \mathcal{E}_{5,N} & = \Big| \mu([n_0+1,\infty])\int_\C f d\theta_\infty - \sum_{n_0 < n \le \infty} \mu(n)\int_\C fd\theta_n\Big|.
\end{align*}
Let us now estimate the terms $\mathcal{E}_{i,N}$.

By independence of $\Gamma_{\Lambda}, \Lambda \in \widehat{G}$, and the definition of the probability measure $\mu_{G_N}$ we have
\begin{align}
\mathcal{E}_{1,N}^2 &= \sum_{\stackrel{\Lambda \in \widehat{G}_N}{\dim \Lambda \le n_0}} \mu_{G_N}(\Lambda)^2 \Var (\int_\C f dL_{\Gamma_\Lambda})
 \le \|f\|_\infty^2 \sum_{\stackrel{\Lambda \in \widehat{G}_N}{\dim \Lambda \le n_0}} \mu_{G_N}(\Lambda)  \frac{(\dim \Lambda)^2}{|G_N|}\nonumber\\
&\le \frac{\|f\|_\infty  n_0^2}{|G_N|} < \varepsilon \label{eq:E_1}
\end{align}
for $N$ large enough.

The term $\mathcal{E}_{2,N}$ converges to zero as $N\to\infty$ since $f$ is bounded, by assumption $\widetilde{\mu}_{G_N}$ converges weakly to $\mu$ and points $1,\ldots,n_0$ are isolated in $\Z_+$ (which implies that $\widetilde{\mu}_{G_N}(k) \to\mu(k)$). Thus for large $N$,
\begin{align}\label{eq:E_2}
\mathcal{E}_{2,N} < \varepsilon.
\end{align}

By inequality \eqref{eq:L_2-distance} and the triangle inequality in $L_2$ we have
\begin{align}\label{eq:E_3}
\mathcal{E}_{3,N} \le \sum_{\stackrel{\Lambda \in \widehat{G}}{\dim(\Lambda) > n_0}}  \mu_{G_N}(\Lambda) \varepsilon \le \varepsilon.
\end{align}

Since the set $[n_0+1,\infty] \subseteq \overline{\Z}_+$ has empty boundary, we get $\widetilde{\mu}_{G_N}([n_0+1,\infty]) \to \mu([n_0+1,\infty])$ and thus for $N$ large enough
\begin{align}\label{eq:E_4}
\mathcal{E}_{4,N} < \varepsilon.
\end{align}

Finally, by \eqref{eq:error-at-infinity}, we get
\begin{align}\label{eq:E_5}
\mathcal{E}_{5,N} \le \sum_{n_0 < n < \infty} \mu(n)|\int_\C fd\theta_n - \int_\C fd\theta_\infty| \le \varepsilon.
\end{align}

Combining \eqref{eq:5-term-decomposition} with \eqref{eq:E_1}--\eqref{eq:E_5} we get that $\mathcal{E}_N \le 5\varepsilon$ for $N$ large enough, which ends the proof of the theorem.
\end{proof}

\begin{remark} Let us note that a straightforward adaptation of the above argument gives the proof of the Gaussian version of Theorem \ref{thm:singular-values}. However, to get it in full generality we will need some additional ingredients, presented in the next section.
\end{remark}

\section{Universality of the limiting distribution of singular values}\label{sec:universality}

We will now prove Theorem \ref{thm:singular-values}. Before we proceed with the details of the argument, let us present its outline. We will again rely on Corollary \ref{cor:deterministic-spectrum-explained} and relate the spectral distribution of $\widehat{X^N}(\Lambda)$ to that of a Ginibre matrix. For small dimensional representations it will be done at the level of joint distributions of entries, simply by means of the Central Limit Theorem. For high dimensional $\Lambda$ however, CLT or Berry-Esseen bounds available in the literature are not strong enough to give a good Gaussian approximation of the distribution of the matrix. Therefore we will use the moment method to show universality at the level of expected spectral measure and combine it with concentration inequalities for traces of polynomials of random matrices due to Meckes and Szarek \cite{MR2869165} to get a bound in probability. To use the moment method and concentration we will need to truncate the random variables $X^N_g$, which can be done by an application of the Hoffman-Wielandt inequality.

\subsection{Representations of small dimension}

Let us start with the part corresponding to the small dimensional representations.
We have the following lemma, inspired by the work of Meckes \cite{MR3069372}.

\begin{lemma}\label{le:finite-dimension-CLT}
Fix a positive integer $n$. Let $G_N$ be a sequence of groups with sizes tending to $\infty$ and let $\Lambda_N$, $\Delta_N$ be two different irreducible representations of $G_N$ of dimension at most $n$. Let $X^N$ be as in Theorem \ref{thm:singular-values} and let $\Gamma_\Lambda$ be defined as in Proposition \ref{prop:Gaussian-distribution}. Let also $f \colon \R_+\to \R$ be a bounded continuous function. Then as $N\to \infty$ we have the convergence
\begin{align}\label{eq:CLT-expectations}
\Big|\E \int_{\R_+} f d L_{\frac{1}{\sqrt{|G_N|}}\sqrt{\widehat{X}^N(\Lambda_N) \widehat{X}^N(\Lambda_N)^\ast}} - \E \int_{\R_+} f d L_{\sqrt{\Gamma_{\Lambda_N} \Gamma_{\Lambda_N}^\ast}}\Big| \to 0,
\end{align}
and
\begin{align}\label{eq:CLT-variances}
\Cov\Big(\int_{\R_+} f d L_{\frac{1}{\sqrt{|G_N|}}\sqrt{\widehat{X}^N(\Lambda_N)\widehat{X}^N(\Lambda_N)^\ast}},\int_{\R_+} f d L_{\frac{1}{\sqrt{|G_N|}}\sqrt{\widehat{X}^N(\Delta_N)\widehat{X}^N(\Delta_N)^\ast}}\Big) \to 0.
\end{align}
\end{lemma}

\begin{proof}
Since $n$ is fixed, by splitting the sequence $G_N$ into a finite number of subsequences, we can assume that $\dim(\Lambda_N) = k$ and $\dim(\Delta_N) = l$ are independent of $N$.

Consider the couple of random matrices $Z_N := (\frac{1}{\sqrt{|G_N|}}\widehat{X}^N(\Lambda_N),\frac{1}{\sqrt{|G_N|}}\widehat{X}^N(\Delta_N))$ as a real random vector in $\R^{2(k^2 + l^2)} \simeq \C^{k^2+l^2}$. By Lemma \ref{le:covariances} this vector has uncorrelated components, moreover the variance of components corresponding to $\Lambda_N$, $\Delta_N$ equals respectively $1/(2k)$ and $1/(2l)$. Since the matrices $\Lambda_N(g), \Delta_N(g)$ are unitary, they are bounded independently of $N$ (we work in fixed dimension, so the choice of norm is irrelevant) and thus the Lindeberg condition for the sums $\frac{1}{\sqrt{|G_N|}}\sum_{g \in G} X_g^N (\Lambda_N(g),\Delta_N(g))$ is trivially satisfied (recall that we work under the assumption that the distribution of $X^N_g$ is independent of $N$ and $g$). Thus $Z_N$ converges weakly to the random vector $(\Gamma^{(1)},\Gamma^{(2)})$ where $\Gamma^{(i)}$ are independent complex Ginibre matrices of size $k$ and $l$ respectively. Note that this is also the distribution of $(\Gamma_{\Lambda_N},\Gamma_{\Delta_N})$. Since eigenvalues of a matrix are continuous functions of the entries (see e.g. \cite[Appendix D]{MR832183}), this clearly implies \eqref{eq:CLT-expectations}. To get \eqref{eq:CLT-variances} note that the function
\begin{displaymath}
(A,B) \mapsto \int f dL_A \int f dL_B,
\end{displaymath}
where $A,B$ are square matrices of size $k,l$ respectively, is bounded and continuous on $\R^{2(k^2 + l^2)}$. Thus
\begin{align*}
&\E\int_{\R_+} f d L_{\frac{1}{\sqrt{|G_N|}}\sqrt{\widehat{X}^N(\Lambda_N)(\widehat{X}^N(\Lambda_N))^\ast}} \int_{\R_+} f d L_{\frac{1}{\sqrt{|G_N|}}\sqrt{\widehat{X}^N(\Delta_N)(\widehat{X}^N(\Delta_N))^\ast}} \\
&\to \E \int_{\R_+} f d L_{\sqrt{\Gamma^{(1)} (\Gamma^{(1)})^\ast}}\int_{\R_+} f d L_{\sqrt{\Gamma^{(2)} (\Gamma^{(2)})^\ast}} = \E \int_{\R_+} f d L_{\sqrt{\Gamma^{(1)} (\Gamma^{(1)})^\ast}}\E \int_{\R_+} f d L_{\sqrt{\Gamma^{(2)} (\Gamma^{(2)})^\ast}},
\end{align*}
which in combination with \eqref{eq:CLT-expectations} proves \eqref{eq:CLT-variances}.
\end{proof}

\subsection{High-dimensional representations}

The next lemma will help us deal with the contribution from high dimensional representations.

\begin{lemma}\label{le:moments-universality} Let $G_N$ be a sequence of groups with sizes tending to infinity and let $\xi$ and $X^N$ be as in Theorem \ref{thm:singular-values}. Assume additionally that $\xi$ has finite moments of all orders. Let $\Lambda_N$ be a sequence of irreducible representations of $G_N$ and denote $d_N := \dim \Lambda_N$. Then for any positive integer $k$ we have
\begin{align}\label{eq:moments-difference}
  \E \frac{1}{d_N} \tr \Big(\frac{1}{|G_N|} \widehat{X^N}(\Lambda_N)\widehat{X^N}(\Lambda_N)^\ast\Big)^k - \E \frac{1}{d_N} \tr ( \Gamma_{\Lambda_N}\Gamma_{\Lambda_N}^\ast)^k \stackrel{N\to \infty}{\to} 0,
\end{align}
where $\Gamma_\Lambda$ is as in Proposition \ref{prop:Gaussian-distribution}.

\end{lemma}

\begin{remark}
Let us stress that in the above lemma we do not assume that the dimension $d_N$ tends to infinity with $N$. This assumption will be needed later to show that the moments of the empirical spectral measure are concentrated around their expectations (see Corollary \ref{cor:concentration-for-traces} below).
\end{remark}
\begin{proof}
  Let $Y^N = (Y_g^N)_{g\in G_N}$ be a family of i.i.d. standard complex Gaussian variables.

  By Lemma \ref{le:covariances} we can assume that
  \begin{displaymath}
    \Gamma_{\Lambda_N} = \frac{1}{\sqrt{|G_N|}} \sum_{g \in G} Y^N_g \Lambda_N(g).
  \end{displaymath}
  We also have
  \begin{displaymath}
  \frac{1}{\sqrt{|G_N|}}\widehat{X^N}(\Lambda_N) = \frac{1}{\sqrt{|G_N|}} \sum_{g \in G} X^N_g \Lambda_N(g).
   \end{displaymath}
Therefore, we get
\begin{align*}
  &\E \frac{1}{d_N} \tr \Big(\frac{1}{|G_N|} \widehat{X^N}(\Lambda_N)\widehat{X^N}(\Lambda_N)^\ast\Big)^k \\
  & =
  \frac{1}{d_N|G_N|^k} \sum_{g_1,\ldots,g_{2k}\in G_N} \Big(\E \prod_{i=1}^k X^N_{g_{2i-1}}\overline{X^N_{g_{2i}}}\Big) \tr \Big(\prod_{i=1}^k \Lambda_N(g_{2i-1})\Lambda_N(g_{2i})^\ast\Big).
\end{align*}
and a similar equality for $\E \frac{1}{d_N} \tr (\Gamma_{\Lambda_N}\Gamma_{\Lambda_N}^\ast)^k$.

We can now proceed as in the classical proof of Wigner's theorem, the argument is in fact simpler. Note that if there is some $g \in G$, which appears in the sequence $(g_j)_{j=1}^{2k}$ exactly once, then by the mean zero and independence assumption, the corresponding summand on the right hand side above vanishes (and the same happens for the expansion of  $\E \frac{1}{d_N} \tr (\Gamma_{\Lambda_N}\Gamma_{\Lambda_N}^\ast)^k$).

On the other hand, there are at most $C_k |G_N|^{k-1}$ sequences in which every $g \in G$ appears either zero times or at least $2$ times and there exists $g \in G$, appearing at least $3$ times ($C_k$ is a constant depending only on $k$). Also, since $\xi$ has finite moments of all orders, by H\"older's inequality $|\E \prod_{i=1}^k X^N_{g_{2i-1}}\overline{X^N_{g_{2i}}}| \le D_k$ for some constant depending only on $k$ and the law of $\xi$. Finally, since the product of matrices in each summand above is unitary, the modulus of its trace is at most $d_N$. Altogether this shows that the contribution to each of the expressions on the left hand side of \eqref{eq:moments-difference} from all such sequences is at most
\begin{displaymath}
  \frac{1}{d_N |G_N|^k} C_k |G_N|^{k-1} D_k d_N = \frac{C_k D_k}{|G_N|} \to 0
\end{displaymath}
as $N \to \infty$.

Again, the same argument applies to $\E \frac{1}{d_N} \tr (\Gamma_{\Lambda_N}\Gamma_{\Lambda_N}^\ast)^k$ and thus asymptotically the left hand side of \eqref{eq:moments-difference} equals
\begin{displaymath}
  \frac{1}{d_N|G_N|^k} \sum_{(g_1,\ldots,g_{2k})\in \mathcal{A}_N} \Big(\E \prod_{i=1}^kX^N_{g_{2i-1}}\overline{X^N_{g_{2i}}} - \E  \prod_{i=1}^k Y^N_{g_{2i-1}}\overline{Y^N_{g_{2i}}}\Big) \tr \Big(\prod_{i=1}^k \Lambda^N(g_{2i-1})\Lambda^N(g_{2i})^\ast\Big),
\end{displaymath}
where $\mathcal{A}_N$ is the set of $G_N$-valued sequences of length $2k$ in which every value is taken exactly twice. Note however that for such sequences the expectation $\E \prod_{i=1}^kX^N_{g_{2i-1}}\overline{X^N_{g_{2i}}}$ and $\E  \prod_{i=1}^k Y^N_{g_{2i-1}}\overline{Y^N_{g_{2i}}}$ either both vanish (if there exist $i\neq j$ such that $g_{2i} = g_{2j}$ or $g_{2i-1} = g_{2j-1})$) or are both equal to one (if no such $i,j$ exist).

Thus the asymptotic value of the left hand side of \eqref{eq:moments-difference} is indeed equal to zero.
\end{proof}

To pass from the analysis of expected spectral measure to a statement in probability, we are going to need a concentration of measure result for random matrices. We will use a theorem due to Meckes and Szarek \cite{MR2869165} which we formulate below in a restricted form, sufficient for our purposes. Recall first that a random $d\times d$  matrix $A$ has the convex concentration property with constant $K$ if for every convex function $f\colon M(d) \to \R$, which is $1$-Lipschitz with respect to the Hilbert-Schmidt norm, and every $t > 0$,
\begin{displaymath}
\p(|f(A) - \Med f(A)| \ge t) \le 4\exp(-t^2/K^2),
\end{displaymath}
where $\Med Z$ denotes the median of a random variable $Z$.

A special case of Theorem 1 of \cite{MR2869165} is as follows (please note that to adapt the result to our needs we changed the notation with respect to the original statement, in which $d$ stands for the degree of the polynomial, and the dimension is denoted by $n$).
\begin{theorem}\label{thm:Meckes-Szarek}
For every positive integer $k$, there exist constants $C_k, c_k$ such that for every $d\times d$ random matrix $A$ satisfying the convex concentration property with constant $1$ and every  $t > 0$,
\begin{displaymath}
\p\Big(\Big|\tr \Big(\frac{AA^\ast}{d}\Big)^k - \E \tr \Big(\frac{AA^\ast}{d}\Big)^k\Big| \ge t\Big) \le C_k \exp\Big(-c_k\min(t^2, d t^{1/k})\Big).
\end{displaymath}
\end{theorem}
Let us now note that if $\Lambda$ is an irreducible representation of $G_N$ of dimension $d$ and the variables $X^N_g$ are bounded by $a$ then the random matrix
\begin{displaymath}
\widehat{X^N}(\Lambda) = \sum_{g \in G} X^N_g \Lambda(g)
\end{displaymath}
has the convex concentration property with constant $2 \sqrt{|G_N|a/d}$.

Indeed, if $f \colon M(d) \to \R$ is a convex function and $x = (x_g)_{g\in G_N}$ then $\widetilde{f}( x) = f(\widehat{x}(\Lambda))$ is convex as a function of $x$.

Applying the Plancherel formula \eqref{eq:Plancherel} to $x$ and $y_g = \overline{x_{g^{-1}}}$, and noting that $\widehat{y}(\Delta) = \sum_{g \in G_N} \bar{x}_g \Delta(g^{-1}) = \widehat{x}(\Delta)^\ast$ we obtain
\begin{align}\label{eq:Lipschitz-constant}
\sum_{g \in G_N} |x_g|^2 = \frac{1}{|G_N|} \sum_{\Delta \in \widehat{G_N}} \dim(\Delta) \tr (\widehat{x}(\Delta)\widehat{x}(\Delta)^\ast) \ge \frac{d}{|G_N|} \|\widehat{x}(\Lambda)\|_{HS}^2.
\end{align}
Thus, if the function $f$ is 1-Lipschitz with respect to the Hilbert-Schmidt norm, then the function $\widetilde{f}$ is $\sqrt{|G_N|/d}$-Lipschitz with respect to the standard Euclidean norm on $\C^{G_N}$. Together with independence of $X^N_g$, $g\in G_N$, by the celebrated Talagrand's concentration
inequality for convex functions of independent bounded variables (\cite{MR1361756}, see \cite[Corollary 4]{MR2722794} for a more general version, encompassing the complex case), this gives that $\widehat{X^N}(\Lambda)$ has the convex concentration property with constant $2 a \sqrt{|G_N|/d}$. Thus the matrix
\begin{displaymath}
A := \frac{\sqrt{d}\widehat{X^N}(\Lambda)}{2a \sqrt{|G_N|}}
\end{displaymath}
has the convex concentration property with constant 1.
We have
\begin{align*}
&\p\Big( \Big|\frac{1}{d} \tr \Big(\frac{\widehat{X^N}(\Lambda)\widehat{X^N}(\Lambda)^\ast}{|G_N|}\Big)^k - \E \frac{1}{d} \tr \Big(\frac{\widehat{X^N}(\Lambda)\widehat{X^N}(\Lambda)}{|G_N|}\Big)^k \Big| \ge t\Big)\\
&=\p\Big( \Big|\tr \Big(\frac{AA^\ast}{d}\Big)^k - \E \tr \Big(\frac{AA^\ast}{d}\Big)^k| \ge \frac{td}{(2a)^{2k}}\Big)
\end{align*}
and thus, by  Theorem \ref{thm:Meckes-Szarek} we obtain
\begin{corollary}\label{cor:concentration-for-traces}
If the variables $X^N_g$ are bounded by $a$, and $\Lambda$ is an irreducible representation of $G_N$ of dimension $d$, then for every positive integer $k$ and every $t> 0$,
\begin{align*}
&\p\Big( \Big|\frac{1}{d} \tr \Big(\frac{\widehat{X^N}(\Delta)\widehat{X^N}^\ast}{|G_N|}\Big)^k - \E \frac{1}{d} \tr \Big(\frac{\widehat{X^N}(\Delta)\widehat{X^N}}{|G_N|}\Big)^k \Big| \ge t\Big) \\
& \le C_k\exp\Big(-c_k\min\Big(\frac{d^2 t^2 }{(2a)^{4k}}, \frac{d^{1+1/k}t^{1/k}}{(2a)^2}\Big)\Big).
\end{align*}
\end{corollary}
Combining this with Lemma \ref{le:moments-universality}, the well-known fact that for all positive integers $k$,
\begin{displaymath}
\frac{1}{n} \E \tr (\Gamma_n \Gamma_n^\ast)^k \stackrel{n \to \infty}{\to} \int_0^\infty x^{2k} d\rho_\infty(x)
\end{displaymath}
(see e.g. \cite[Chapter 3.2]{MR2567175}) and Lemma \ref{le:convergence-moments} from the Appendix, we obtain

\begin{prop}\label{prop:bounded-matrices-high-dimension} Let $G_N, \xi$ and $X^N$ be as in Theorem \ref{thm:singular-values}. Assume additionally that $\xi$ is bounded. Let $\Lambda_N$ be a sequence of irreducible representations of $G_N$ with $d_N := \dim \Lambda_N \to \infty$ as $N \to \infty$. Then for any positive integer $k$ we have
\begin{displaymath}
\frac{1}{d_N} \tr \Big(\frac{1}{|G_N|} \widehat{X^N}(\Lambda_N)\widehat{X^N}(\Lambda_N)^\ast\Big)^k \to  \int_0^\infty x^{2k} d\rho_\infty(x)
\end{displaymath}
in probability as $N \to \infty$.
In particular $L_{\frac{1}{\sqrt{|G_N|}}\sqrt{\widehat{X^N}(\Lambda_N)\widehat{X^N}(\Lambda_N)^\ast}}$ converges weakly in probability to $\rho_\infty$.
\end{prop}

Our next goal is to remove the boundedness assumption on $\xi$. This can be done by a truncation and recentering procedure as for classical Wigner random matrices.

Consider arbitrary $M > 0$ and the random variable $\xi' = \xi \ind{|\xi| \le M}$ seen as a random vector in $\R^2$ and let $C_M$ be its covariance matrix. Note that since $Cov(\xi) = \frac{1}{2}\id_2$ (where $\id_2$ is the $2\times 2$ identity matrix), we have
\begin{align}\label{eq:covariance-proximity}
\lim_{M \to \infty} C_M = \frac{1}{2}\id_2,
\end{align}
in particular for sufficiently large $M$, the matrix $C_M$ is non-singular.

Let now $\xi'' = \frac{1}{\sqrt{2}}C_M^{-1/2} (\xi' - \E \xi')$ so that $\xi''$ as a complex random variable satisfies $\E \xi'' = \E (\xi'')^2 = 0$, $\E |\xi''|^2 = 1$.
Let moreover for $N \ge 1$ and $g \in G_N$, $Y^{N,M}_g$ be i.i.d. copies of $\xi''$.


By Proposition \ref{prop:bounded-matrices-high-dimension} we get that $L_{\frac{1}{\sqrt{|G_N|}} \sqrt{\widehat{Y^{N,M}}(\Lambda_N)\widehat{Y^{N,M}}(\Lambda_N)^\ast}}$ converges weakly in probability to $\rho_\infty$. In particular for every sufficiently large $M$,
\begin{align}\label{eq:bounded-convergence}
\D\Big(L_{\frac{1}{\sqrt{|G_N|}} \sqrt{\widehat{Y^{N,M}}(\Lambda_N)\widehat{Y^{N,M}}(\Lambda_N)^\ast}}, \rho_\infty\Big) \to 0
\end{align}
in probability as $N \to \infty$, where $\D$ is the L\'{e}vy distance between probability measures, i.e.
\begin{displaymath}
  \D(\mu,\nu) = \inf\{\varepsilon > 0 \colon \forall_{x \in \R}\; \mu((-\infty,x-\varepsilon))-\varepsilon \le \nu((-\infty,x)) \le \mu((-\infty,x+\varepsilon))+\varepsilon\}.
\end{displaymath}

 By Theorems A.37 and A.38 from \cite{MR2567175}
 for any two square matrices $A$, $B$ of dimension $d_N$,
\begin{align}\label{eq:HW}
\D(L_{\sqrt{AA^\ast}},L_{\sqrt{BB^\ast}})^3 \le \frac{1}{d_N} \|A - B\|_{HS}^2.
\end{align}
Using the equality $\E X^N_g = 0$, we can write
\begin{align*}
& \E \frac{1}{d_N}\Big\| \frac{1}{\sqrt{|G_N|}} \widehat{X^N}(\Lambda) - \frac{1}{\sqrt{|G_N|}}  \widehat{Y^{N,M}}(\Lambda)\Big\|_{HS}^2 \\
=& \frac{1}{d_N|G_N|}\E \Big\| \sum_{g \in G_N} (X^N_g - Y^{N,M}_g)\Lambda(g)\Big\|_{HS}^2 \\
=&  \frac{1}{d_N|G_N|} \E \Big\|\sum_{g\in G_N}(X^N_g\ind{|X^N_g| \le M} - \E X^N_g\ind{|X^N_g| \le M} - Y^{N,M}_g)\Lambda(g) \\
&\phantom{aaaaaaaaaa}+ \sum_{g \in G_N} (X^N_g\ind{|X^N_g| > M} - \E X^N_g\ind{|X^N_g| > M})\Lambda(g)\Big\|_{HS}^2\\
\le &  \frac{2}{d_N|G_N|} \Big(\E \Big\|\sum_{g\in G_N}(X^N_g\ind{|X^N_g| \le M} - \E X^N_g\ind{|X^N_g| \le M} - Y^{N,M}_g)\Lambda(g)\Big\|_{HS}^2 \\
&\phantom{aaaaadddd}+ \Big\|\sum_{g \in G_N} (X^N_g\ind{|X^N_g| > M} - \E X^N_g\ind{|X^N_g| > M})\Lambda(g)\Big\|_{HS}^2 \Big)\\
=&  \frac{2}{d_N|G_N|} \Big(\sum_{g\in G_N} \E|(X^N_g\ind{|X^N_g| \le M} - \E X^N_g\ind{|X^N_g| \le M} - Y^{N,M}_g)|^2 \|\Lambda(g)\|_{HS}^2 \\
&\phantom{aaaaadddd}+ \sum_{g\in G_N} \E |X^N_g\ind{|X^N_g| > M} - \E X^N_g\ind{|X^N_g| > M}|^2 \|\Lambda(g)\|_{HS}^2\Big) \\
 = &2 \E |(\xi' -\E \xi') - \xi''|^2 + 2 \E |\xi\ind{|\xi| > M} - \E \xi \ind{|\xi| > M}|^2\\
 \le&  2\E | (\id_2 - (2C_M)^{-1/2})(\xi'-\E \xi')|^2 + 2\E |\xi|^2\ind{|\xi| > M},
\end{align*}
where in the first inequality we used the estimate $(a+b)^2 \le 2 (a^2 + b^2)$ and the triangle inequality, in the third equality the parallelogram  identity for the Hilbert-Schmidt norm and in the last equality the fact that $\Lambda(g)$ are unitary.
Note that by \eqref{eq:covariance-proximity}, for any $\varepsilon > 0$ and $M$ large enough, for all $z \in \C\simeq \R^2$ we have
\begin{displaymath}
|(\id_2 - (2C_M)^{-1/2}) z|^2 \le \varepsilon |z|^2,
\end{displaymath}
hence the first summand on the right hand side above is bounded by $2\varepsilon \E |\xi'|^2 \le 2\varepsilon$. The second summand for sufficiently large $M$ does not exceed $2\varepsilon$. Thus
\begin{displaymath}
  \lim_{M \to \infty} \sup_{N \ge 1}\E \frac{1}{d_N}\Big\| \frac{1}{\sqrt{|G_N|}} \widehat{X^N}(\Lambda) - \frac{1}{\sqrt{|G_N|}}  \widehat{Y^{N,M}}(\Lambda)\Big\|_{HS}^2 = 0,
\end{displaymath}
which combined with \eqref{eq:HW} and Markov's inequality gives for any $\varepsilon > 0$,
\begin{align*}
& \limsup_{M \to \infty} \sup_{N \ge 1} \p\Big(\D\Big(L_{\frac{1}{\sqrt{|G_N|}} \sqrt{\widehat{X^N}(\Lambda)\widehat{X^N}(\Lambda)^\ast}},L_{\frac{1}{\sqrt{|G_N|}} \sqrt{\widehat{Y^{N,M}}(\Lambda)\widehat{Y^{N,M}}(\Lambda)^\ast}}\Big) \ge \varepsilon\Big) \\
&\le \varepsilon^{-3} \lim_{M \to \infty}\sup_{N \ge 1}  \E \frac{1}{d_N}\Big\| \frac{1}{\sqrt{|G_N|}} \widehat{X^N}(\Lambda) - \frac{1}{\sqrt{|G_N|}}  \widehat{Y^{N,M}}(\Lambda)\Big\|_{HS}^2 = 0.
\end{align*}

Combining this with \eqref{eq:bounded-convergence} we get  the following strengthening of Proposition \ref{prop:bounded-matrices-high-dimension}.

\begin{prop}\label{prop:unbounded-matrices-high-dimension}
Let $G_N, \xi$ and $X^N$ be as in Theorem \ref{thm:singular-values} and let $\Lambda_N$ be a sequence of irreducible representations of $G_N$ with $d_N := \dim \Lambda_N \to \infty$ as $N \to \infty$. Then $L_{\frac{1}{\sqrt{|G_N|}}\sqrt{\widehat{X^N}(\Lambda_N)\widehat{X^N}(\Lambda_N)^\ast}}$ converges weakly in probability to $\rho_\infty$.
\end{prop}

\subsection{Proof of Theorem \ref{thm:singular-values}} We are now ready to conclude the proof of Theorem \ref{thm:singular-values}, by combining the asymptotic behaviour for small-dimensional representations, given in Lemma \ref{le:finite-dimension-CLT} with the asymptotics for high-dimensional representations governed by Proposition \ref{prop:unbounded-matrices-high-dimension}.

\begin{proof}[Conclusion of the proof of Theorem \ref{thm:singular-values}]
Let us fix a bounded continuous function $f \colon \R \to \R$ and $\varepsilon > 0$. There exist positive integers $n_0,N_0$ such that for any $n > n_0$ and $N > N_0$, if $|G_N|> N_0$ and $\Lambda$ is an irreducible representation of $G_N$ with dimension $n$, then
\begin{align}\label{eq:high-dimensional-contribution}
  \Big\| \int_\R f d L_{\frac{1}{\sqrt{|G_N|}} \sqrt{\widehat{X^N}(\Lambda)\widehat{X^N}(\Lambda)^\ast}} - \int_\R f d\rho_\infty\Big\|_2 < \varepsilon.
\end{align}
Indeed, if this was not the case, we could find a sequence of groups $G_N$ and their irreducible representations $\Lambda_N$ with $|G_N|,
\dim \Lambda_N \to \infty$ such that the left hand side of the above inequality remains separated from zero. However in view of Lebesgue's theorem on dominated convergence this would contradict Proposition \ref{prop:unbounded-matrices-high-dimension}.
Thanks to the Marchenko-Pastur theorem, by increasing $n_0$ if necessary we can also assume that  for $n > n_0$,
        \begin{align}\label{eq:error-at-infinity-sing}
            \Big|\int_{\R} f d\rho_n - \int_\R f d\rho_\infty\Big| < \varepsilon.
        \end{align}

By Corollary \ref{cor:deterministic-spectrum-explained} we can thus write
\begin{align*}
  \mathcal{E}_N := &\Big \|\int_\R f d L_{\frac{1}{\sqrt{|G_N|}} \sqrt{P_{X_N} P_{X_N}^\ast}} - \int_\R f dL^s_\infty\Big\|_2 \\
   = & \Big\|\sum_{\Lambda \in \hat{G}_N} \mu_{G_N}(\Lambda)\int_{\R} f dL_{\frac{1}{\sqrt{|G_N|}} \sqrt{\widehat{X_N}(\Lambda)\widehat{X_N}(\Lambda)^\ast}}  - \sum_{n \in \Z_+} \mu(n) \int_\R f d \rho_n\Big\|_2 \\
\le & \sum_{k=1}^5 \mathcal{E}_{k,N},
\end{align*}
where
\begin{align*}
  \mathcal{E}_{1,N}& = \Big\|\sum_{\stackrel{\Lambda \in \widehat{G}_N}{\dim \Lambda \le n_0}} \mu_{G_N}(\Lambda)
  \Big(\int_\R f d L_{\frac{1}{\sqrt{|G_N|}} \sqrt{\widehat{X_N}(\Lambda)\widehat{X_N}(\Lambda)^\ast}} - \int_\R fd\rho_{\dim \Lambda}\Big)\Big\|_2,\\
  \mathcal{E}_{2,N} & =  \Big|\sum_{n=1}^{n_0}(\widetilde{\mu}_{G_N}(n) - \mu(n))\int_\R fd\rho_n\Big| ,\\
  \mathcal{E}_{3,N} &= \Big\|\sum_{\stackrel{\Lambda \in \widehat{G}}{\dim(\Lambda) > n_0}}  \mu_{G_N}(\Lambda)
  \Big(\int_\R f dL_{\frac{1}{\sqrt{|G_N|}} \sqrt{\widehat{X_N}(\Lambda)\widehat{X_N}(\Lambda)^\ast}} -  \int_\R f d\rho_\infty\Big)\Big\|_2,\\
  \mathcal{E}_{4,N} &= \Big|(\widetilde{\mu}_{G_N}([n_0+1,\infty]) - \mu([n_0+1,\infty])) \int_\R f d\rho_\infty \Big|,\\
  \mathcal{E}_{5,N} & = \Big| \mu([n_0+1,\infty])\int_\R f d\rho_\infty - \sum_{n_0 < n \le \infty} \mu(n)\int_\R fd\rho_n\Big|.
\end{align*}

As in the proof of Theorem \ref{thm:eigenvalues}, $\mathcal{E}_{2,N}$ and $\mathcal{E}_{4,N}$ converge to zero since $\widetilde{\mu}_{G_N}$ converges weakly to $\mu$ and the sets $\{n\}$, $n \le n_0$ and $[n_0+1,\infty]$ have empty boundary in $\overline{\Z}_+$. By \eqref{eq:error-at-infinity-sing} the term $\mathcal{E}_{5,N}$ does not exceed $\varepsilon$.  Similarly for large $N$ the term $\mathcal{E}_{3,N}$ is smaller than $\varepsilon$ by \eqref{eq:high-dimensional-contribution}.

It thus remains to estimate the term $\mathcal{E}_{1,N}$, which can be easily done by means of Lemma \ref{le:finite-dimension-CLT}. Indeed, we have the equality
\begin{align}\label{eq:expectation-equality}
  \int_{\R} f d\rho_{\dim \Lambda}  = \E \int_{\R} f d L_{\sqrt{\Gamma_{\Lambda}\Gamma_{\Lambda}^\ast}},
\end{align}
moreover, by Lemma \ref{le:finite-dimension-CLT}, we have
\begin{align}\label{eq:small-dimensional-expectation}
  \lim_{N\to \infty} \max_{\stackrel{\Lambda \in \widehat{G}_N}{\dim \Lambda \le n_0}}  \Big|\E \int_{\R_+} f d L_{\frac{1}{\sqrt{|G_N|}}\sqrt{\widehat{X}^N(\Lambda) \widehat{X}^N(\Lambda)^\ast}} - \E \int_{\R_+} f d L_{\sqrt{\Gamma_{\Lambda} \Gamma_{\Lambda}^\ast}}\Big| = 0
\end{align}
and
\begin{align}\label{eq:small-dimensional-covariance}
  \lim_{N \to \infty}   \max_{\stackrel{\Lambda\neq \Delta \in \widehat{G}_N}{\dim \Lambda, \dim \Delta \le n_0}} \Cov\Big(\int_{\R_+} f d L_{\frac{1}{\sqrt{|G_N|}}\sqrt{\widehat{X}^N(\Lambda)\widehat{X}^N(\Lambda)^\ast}},\int_{\R_+} f d L_{\frac{1}{\sqrt{|G_N|}}\sqrt{\widehat{X}^N(\Delta)(\widehat{X}^N(\Delta))^\ast}}\Big) \to 0.
\end{align}
By \eqref{eq:expectation-equality} and \eqref{eq:small-dimensional-expectation} to show that $\mathcal{E}_{1,N} \to 0$ it is enough to prove that
\begin{displaymath}
  \Var\Big(\sum_{\stackrel{\Lambda \in \widehat{G}_N}{\dim \Lambda \le n_0}} \mu_{G_N}(\Lambda)
  \int_\R f d L_{\frac{1}{\sqrt{|G_N|}} \sqrt{\widehat{X_N}(\Lambda)\widehat{X_N}(\Lambda)^\ast}}\Big) \to 0.
\end{displaymath}
Expanding the variance into sum of covariances and using \eqref{eq:small-dimensional-covariance}, we can see that for any $\delta > 0$ if $N$ is large enough, the variance is dominated by
\begin{align*}
 &\sum_{\stackrel{\Lambda \in \widehat{G}_N}{\dim \Lambda \le n_0}} \mu_{G_N}(\Lambda)^2 \Var\Big(\int_\R f d L_{\frac{1}{\sqrt{|G_N|}} \sqrt{\widehat{X_N}(\Lambda)\widehat{X_N}(\Lambda)^\ast}}\Big) + \sum_{\stackrel{\Lambda\neq \Delta \in \widehat{G}_N}{\dim \Lambda, \dim \Delta \le n_0}} \mu_{G_N}(\Lambda)\mu_{G_N}(\Delta)\delta\\
 & \le \|f\|_\infty \sum_{\stackrel{\Lambda \in \widehat{G}_N}{\dim \Lambda \le n_0}} \mu_{G_N}(\Lambda)^2 + \delta,
\end{align*}
where we again used that $\mu_{G_N}$ is a probability measure.
Now
\begin{displaymath}
\sum_{\stackrel{\Lambda \in \widehat{G}_N}{\dim \Lambda \le n_0}} \mu_{G_N}(\Lambda)^2 \le \frac{n_0^2}{|G_N|}\sum_{\stackrel{\Lambda \in \widehat{G}_N}{\dim \Lambda \le n_0}} \mu_{G_N}(\Lambda) \le \frac{n_0^2}{|G_N|} \to 0
\end{displaymath}
as $N \to \infty$, which shows that $\mathcal{E}_{1,N}$ vanishes in the limit, proving the theorem.

\end{proof}

\subsection{Proof of Theorem \ref{thm:singular-values-uniform version}}

Assume that a number $N$ as in the statement of the theorem does not exist. Thus there exists a sequence of groups $G_N$ such that $|G_N|\to \infty$ and for all $N$,
\begin{displaymath}
\p(d(L^s_{G_N},\mathcal{L}_{G_N}) > \varepsilon) > \delta.
\end{displaymath}
We can also assume that $\widetilde{\mu}_{G_N}$ converges to some measure $\mu$ on $\overline{\Z}_+$ (otherwise we may pass to a subsequence).
Thus $L^s_{G_N}$ converges weakly in probability to $L^s_\infty$ -- the measure given by Theorem \ref{thm:singular-values}, which implies that for large $N$, $d(\mathcal{L}_{G_N}, L^s_\infty) > \varepsilon/2$. Using the weak convergence of $\rho_n$ to $\rho_\infty$ it is however easy to prove that $\mathcal{L}_{G_N}$ converges weakly to $L^s_
\infty$, which yields a contradiction ending the proof of the theorem.

\section{Asymptotic freeness}\label{sec:freeness}

We will now restrict our attention to the case when the limiting measure $\mu$ is the Dirac mass at $\infty$ and as a consequence the limiting spectral measure is the circular law. We will consider several independent sequences $X^{N,1}= (X^{N,1}_g)_{g \in G_N},\ldots,X^{N,m} = (X^{N,m}_g)_{g \in G_N}$ of independent random variables and we will consider convergence of the family of matrices $(P_{X^{N,1}},\ldots,P_{X^{N,m}})$ in $\ast$-moments and its asymptotic freeness.

To formulate our results we need to recall some basic notions of Voiculescu's free probability theory. We refer to \cite{MR1746976,MR3585560,MR2760897} for the details on non-commutative probability and its importance in the context of the study of von Neumann algebras and random matrices, here we only briefly introduce the minimum set of ideas necessary for our purposes.

First, a $\C^\ast$-probability space is a couple $(\mathcal{A},\varphi)$, where $\mathcal{A}$ is a unital $C^\ast$-algebra and $\mathcal{\varphi}$ is a state, i.e. a linear functional such that $\varphi(\I) = 1$, where $\I$ is the unit of $\mathcal{A}$, $\varphi(a^\ast) = \overline{\varphi(a)}$ and $\varphi(aa^\ast) \ge 0$ for any $a \in \mathcal{A}$.

For a family $\{a_1,\ldots,a_m \in \mathcal{A}\}$ we define its distribution as the functional $\mu = \mu_{a_1,\ldots,a_m} \colon \C\langle x_1,\ldots,x_m\rangle \to \C$, where $\C\langle x_1,\ldots,x_m\rangle$ is the algebra of non-commutative polynomials in variables $x_i$, given by the formula
\begin{displaymath}
  \mu(Q) = \varphi(Q(a_1,\ldots,a_m))
\end{displaymath}
for $Q \in \C\langle x_1,\ldots,x_m\rangle$.

An element $a \in \mathcal{A}$ is called semicircular if $\mu(a^n) = 0$ for $n$ odd and equals the $m$-th Catalan number $\frac{1}{m+1}\binom{2m}{m}$ for $n = 2m$, $m \in \N$.
Finally, the subalgebras $\mathcal{A}_1,\ldots,\mathcal{A}_m$ are free if for every $n$, every sequence $i_1,\ldots,i_n \in \{1,\ldots,m\}$, such that $i_k \neq i_{k+1}$ for $k=1,\ldots,n-1$, and every sequence $a_k \in \mathcal{A}_{i_k}$, we have
\begin{displaymath}
  \varphi(a_1,\ldots,a_n) = 0.
\end{displaymath}

A family of subsets of $\mathcal{A}$ is called free if the unital subalgebras they generate are free in the above sense. An element $c \in \mathcal{A}$ is called \emph{circular} if it is of the form $c = (a+\sqrt{-1}b)/\sqrt{2}$ for two free semicircular elements $a,b$.

Given a sequence of $\C^\ast$-probability spaces $(\mathcal{A}_N,\varphi_N)$ and a sequence of $m$-tuples $A_N = (a^N_1,\ldots,a^N_m)$, $N \ge 1$ together with an $m$-tuple $A=(a_1,\ldots,a_m)$ of elements of $\mathcal{A}$, we say that $A_N$ converges to $A$ in moments, if for all noncommutative polynomials in the variables $x_1,\ldots,x_m$,
\begin{displaymath}
  \lim_{N\to \infty} \varphi_N(A_N) = \varphi(A).
\end{displaymath}

Below we will use this notion of convergence, treating random $n\times n$ matrices as elements of the noncommutative probability space $(\mathcal{M}_n,\frac{1}{n}\E \tr)$, where $\mathcal{M}_n$ is the set of $M(n)$ valued random variables.

The first result of this section is the following

\begin{prop}\label{prop:freeness}
  Let  $G_N$ be a sequence of finite groups with $|G_N| \to \infty$. Assume that the sequence of measures $\widetilde{\mu}_{G_N}$ converges weakly to Dirac's mass at infinity and let $X^{N,1}= (X^{N,1}_g)_{g \in G_N},\ldots,X^{N,m} = (X^{N,m}_g)_{g \in G_N}$ be independent families of independent standard complex Gaussian variables. Then the family of matrices $(|G_N|^{-1/2}P_{X^{N,1}},|G_N|^{-1/2}P_{X^{N,1}}^\ast, \ldots,|G_N|^{-1/2}P_{X^{N,m}},|G_N|^{-1/2}P_{X^{N,m}}^\ast)$ converges in moments to the family $(c_1,c_1^\ast,\ldots,c_m,c_m^\ast)$, where $c_1,\ldots,c_m$ are circular elements such that $\{c_1,c_1^\ast\}$,$\ldots$, $\{c_m,c_m^\ast\}$ are free.
\end{prop}

\begin{remark}
The interest in the above theorem stems from two reasons. First, it allows to calculate the asymptotic mean eigenvalue distribution of Hermitian matrices created with use of $(P_{X^{N,i}})_{i \le m}$, for instance it shows that the mean eigenvalue distribution of $\frac{1}{\sqrt{2|G_N|}}(P_{X^{N,1}} + P_{X^{N,1}}^\ast)$ converges to the Wigner's semicircular law while for $\frac{1}{|G_N|}(P_{X^{N,1}}P_{X^{N,1}}^\ast + P_{X^{N,2}}P_{X^{N,2}}^\ast)$ the limit is the free Poisson (Marchenko-Pastur) distribution with parameter $2$, i.e. the measure with density $f(x)= (2\pi x)^{-1}\sqrt{(x-a_-)(a_+ -x)}\ind{[a_-,a_+]}(x)$, where  $a_\pm = (1\pm\sqrt{2})^2$.

Second, it provides a natural class of examples of asymptotically free structured matrices which are not unitarily invariant and have dependent coefficients, which is in some contrast with the usual examples known from the literature.
\end{remark}

\begin{remark}
  Note that in the case of non-Hermitian matrices, being $\ast$-polynomial functions of the sequence $P_{X^{N,i}}$, deriving limiting spectral distribution from asymptotic freeness is a more subtle issue due to lack of implications between moments convergence and weak convergence of spectral measure (an issue which has been the reason of difficulties in the proof of the circular law for random matrices, see \cite{MR2908617} for an extensive discussion and an overview of the history of this problem).
\end{remark}
\begin{remark}
If one assumes more about the speed of convergence of $\widehat{\mu}_{G_N}$ to $\delta_\infty$, then using the full strength of concentration results for polynomials in random matrices due to Meckes and Szarek \cite{MR2869165} one should be able to improve the above results to asymptotic freeness almost everywhere (see \cite{MR1746976,MR3585560} for precise definition). It is however easy to see that for general sequences of groups one cannot hope for such results and therefore we are not going to pursue this direction here.
\end{remark}

\begin{proof}[Proof of Proposition \ref{prop:freeness}] Our goal is to prove that for every noncommutative polynomial $Q$ in variables $x_1,\ldots,x_{2m}$,
\begin{align}\label{eq:freeness-to-be-proven}
  &\frac{1}{|G_N|}\E \tr Q(|G_N|^{-1/2}P_{X^{N,1}},|G_N|^{-1/2}P_{X^{N,1}}^\ast,\ldots,|G_N|^{-1/2}P_{X^{N,m}},|G_N|^{-1/2}P_{X^{N,m}}^\ast) \nonumber \\
  &\to Q(c_1,c_1^\ast,\ldots,c_m,c_m^\ast).
\end{align}
By linearity it is clearly enough to consider monomials. Let thus $Q(x_1,\ldots,x_{2m}) = x_{i_1}\cdots x_{i_{d}}$ for some positive integer $d$ and $i_1,\ldots,i_d \in \{1,\ldots,2m\}$. To simplify the notation for $g \in G_N$ denote also $Y^{N,i}_g = \frac{1}{\sqrt{|G_N|}} X^{N,(i+1)/2}_g$ for $i$ odd and $Y^{N,i} = \frac{1}{\sqrt{|G_N|}} \overline{X^{N,i/2}_{g^{-1}}}$ for $i$ even. Note that $\frac{1}{\sqrt{|G_N|}} P_{X^{N,i}}^\ast = P_{Y^{N,2i}}$. We have
\begin{align*}
&\frac{1}{|G_N|} \tr Q(\frac{1}{\sqrt{|G_N|}}P_{X^{N,1}},\frac{1}{\sqrt{|G_N|}}P_{X^{N,1}}^\ast,\ldots,\frac{1}{\sqrt{|G_N|}}P_{X^{N,m}},\frac{1}{\sqrt{|G_N|}}P_{X^{N,m}}^\ast) \\
&= \int_\C x dL_{P_{Y^{N,i_1}}\cdots P_{Y^{N,i_d}}} = \int_\C x dL_{P_{Y^{N,i_1}\ast \cdots\ast Y^{N,i_d}}}.
\end{align*}

By Corollary \ref{cor:deterministic-spectrum-explained} and the product property of the Fourier transform, the right-hand side above equals
\begin{displaymath}
  \sum_{\Lambda \in \widehat{G_N}} \mu_{G_N}(\Lambda) \int_C x dL_{\widehat{Y^{N,i_1}}(\Lambda)\cdots \widehat{Y^{N,i_d}}(\Lambda)} =
  \sum_{\Lambda \in \widehat{G_N}} \mu_{G_N}(\Lambda) \frac{1}{\dim (\Lambda)} \tr \widehat{Y^{N,i_1}}(\Lambda)\cdots \widehat{Y^{N,i_d}}(\Lambda).
\end{displaymath}
Thus, the left hand side of \eqref{eq:freeness-to-be-proven} equals
\begin{align}\label{eq:freeness-Fourier-representation}
  \sum_{\Lambda \in \widehat{G_N}} \mu_{G_N}(\Lambda) \frac{1}{\dim (\Lambda)} \E \tr \widehat{Y^{N,i_1}}(\Lambda)\cdots \widehat{Y^{N,i_d}}(\Lambda).
\end{align}
By Lemma \ref{le:covariances} denoting $\dim \Lambda = n$, $\widehat{Y^{N,1}}(\Lambda),\ldots,\widehat{Y^{N,2m}}(\Lambda)$ have the same joint distribution as $\Gamma_{n,1}$, $\Gamma_{n,1}^\ast$,$\ldots$, $\Gamma_{n,m}$, $\Gamma_{n,m}^\ast$, where $\Gamma_{n,i}, i\le m$ are i.i.d. Ginibre matrices of size $n\times n$. Since it is known (see Corollary 4.3.8 and the discussion on p. 147 in \cite{MR1746976}) that the latter sequence converges in moments to $(c_1,c_1^\ast,\ldots,c_m,c_m^\ast)$, we get that for every $\varepsilon > 0$, there exists $n_0 > 0$ such that for all $N$,
\begin{align*}
\Big|\sum_{{\Lambda \in \widehat{G_N}}\atop{\dim \Lambda > n_0}} \E \mu_{G_N}(\Lambda) \frac{1}{\dim (\Lambda)} \tr \widehat{Y^{N,i_1}}(\Lambda)\cdots \widehat{Y^{N,i_d}}(\Lambda) - \widetilde{\mu}_{G_N}([n_0,\infty]) \varphi(c_1,c_1^\ast,\ldots,c_m,c_m^\ast)\Big| \le \varepsilon.
\end{align*}
Using the convergence of $\widetilde{\mu}_{G_N}$ to Dirac's mass at $\infty$, we conclude that for $N$ large enough,
\begin{displaymath}
\Big|\sum_{{\Lambda \in \widehat{G_N}}\atop{\dim \Lambda > n_0}} \E \mu_{G_N}(\Lambda) \frac{1}{\dim (\Lambda)} \tr \widehat{Y^{N,i_1}}(\Lambda)\cdots \widehat{Y^{N,i_d}}(\Lambda) - \varphi(c_1,c_1^\ast,\ldots,c_m,c_m^\ast)\Big| \le 2\varepsilon.
\end{displaymath}
It thus remains to show that the contribution to \eqref{eq:freeness-Fourier-representation} from small-dimensional representations is negligible.
Note however that if $\dim n \le n_0$, then
\begin{displaymath}
\E \tr \widehat{Y^{N,i_1}}(\Lambda)\cdots \widehat{Y^{N,i_d}}(\Lambda)
\end{displaymath}
is the expectation of one of at most $C(n_0,d)$ polynomials in i.i.d. standard complex Gaussian variables and their complex conjugates, where $C(n_0,d)$ depends only on $n_0$ and $d$. Since Gaussian variables have all moments, there exists a constant $C'(n_0,d)$ such that
\begin{displaymath}
\max_{N,\Lambda \in \widehat{G_N}, \dim \Lambda \le n_0} \E \tr \widehat{Y^{N,i_1}}(\Lambda)\cdots \widehat{Y^{N,i_d}}(\Lambda) \le C'(n_0,d).
\end{displaymath}
Thus
\begin{displaymath}
  |\sum_{{\Lambda \in \widehat{G_N}}\atop{\dim \Lambda \le n_0}} \E \mu_{G_N}(\Lambda) \frac{1}{\dim (\Lambda)} \tr \widehat{Y^{N,i_1}}(\Lambda)\cdots \widehat{Y^{N,i_d}}(\Lambda)| \le C'(n_0,d) \widetilde{\mu}_{G_N}([1,n_0]) ,
\end{displaymath}
which converges to zero as $N \to \infty$. This proves the theorem.
\end{proof}

We can now extend the result on freeness to more general distributions.

\begin{theorem}\label{thm:freeness-nongaussian} Let $G_N$ be as sequence of finite groups with $|G_N|\to \infty$. Assume that the sequence of measures $\widetilde{\mu}_{G_N}$ converges weakly to Dirac's mass at infinity. Let $\xi$ be a complex random variable such that $\E \xi = \E \xi^2 = 0$, $\E|\xi|^2 = 1$ and $\xi$ has finite moments of all orders.
 Let also $X^{N,1}= (X^{N,1}_g)_{g \in G_N},\ldots,X^{N,m} = (X^{N,m}_g)_{g \in G_N}$ be independent families of independent copies of $\xi$. Then the family of random matrices $(|G_N|^{-1/2}P_{X^{N,1}},|G_N|^{-1/2}P_{X^{N,1}}^\ast, \ldots,|G_N|^{-1/2}P_{X^{N,m}},|G_N|^{-1/2}P_{X^{N,m}}^\ast)$ converges in moments to the family $(c_1,c_1^\ast,\ldots,c_m,c_m^\ast)$, where $c_1,\ldots,c_m$ are circular elements such that $\{c_1,c_1^\ast\}$,$\ldots$, $\{c_m,c_m^\ast\}$ are free.
\end{theorem}

The above result is an immediate corollary to Proposition \ref{prop:freeness} and the following generalization of Lemma \ref{le:moments-universality} (note that it is valid in greater generality than needed just for Theorem \ref{thm:freeness-nongaussian}).

\begin{lemma}\label{le:universality-polynomials}
Let $G_N$ be as sequence of finite groups with $|G_N|\to \infty$. Let $\xi$ be a complex random variable such that $\E \xi = \E \xi^2 = 0$, $\E|\xi|^2 = 1$ and $\xi$ has finite moments of all orders.
 Let also $X^{N,1}= (X^{N,1}_g)_{g \in G_N},\ldots,X^{N,m} = (X^{N,m}_g)_{g \in G_N}$ be independent families of i.i.d. copies of $\xi$ and $Y^{N,1}= (Y^{N,1}_g)_{g \in G_N},\ldots,Y^{N,m} = (Y^{N,m}_g)_{g \in G_N}$ be independent families of i.i.d. copies of a standard complex Gaussian variable. Let $Q$ be a non-commutative polynomial in the variables $x_1,\ldots,x_{2m}$, and define
 \begin{align*}
   A_N &= \frac{1}{|G_N|}\E \tr Q\Big(\Big(\frac{1}{\sqrt{|G_N|}} P_{X^{N,i}},\frac{1}{\sqrt{|G_N|}} P_{X^{N,i}}^\ast\Big)_{i \le m}\Big),\\
 B_N & = \frac{1}{|G_N|}\E \tr Q\Big(\Big(\frac{1}{\sqrt{|G_N|}} P_{Y^{N,i}},\frac{1}{\sqrt{|G_N|}} P_{Y^{N,i}}^\ast\Big)_{i \le m}\Big).
 \end{align*}
 Then $A_N - B_N \to 0$ as $N \to \infty$.
\end{lemma}

\begin{proof}
The proof is based on the same strategy as in the case of Lemma \ref{le:moments-universality}, namely one shows that the asymptotic behaviour of mixed moments depends only on the covariance structure of the underlying random sequences. We provide the details for the sake of completeness.

Without loss of generality we can assume that $Q$ is a noncommutative monomial, specifically that $Q(x_1,\ldots,x_{2m}) = x_{i_1}\cdots x_{i_{n}}$ for some positive integer $n$ and $i_1,\ldots,i_n \in \{1,\ldots,2m\}$.

Denote for any odd $i \le 2m$, the sequences $S^{N,i}_g = \frac{1}{\sqrt{|G_N|}} X^{N,(i+1)/2}_g$ and $Z^{N,i} = \frac{1}{\sqrt{|G_N|}} Y^{N,(i+1)/2}_g$ and for $i\le m$ even
$S^{N,i}_g = \frac{1}{\sqrt{|G_N|}} \overline{X^{N,i/2}_{g^{-1}}}$ and $Z^{N,i}_g = \frac{1}{\sqrt{|G_N|}} \overline{Y^{N,i/2}_{g^{-1}}}$.
Our goal is to prove that
\begin{displaymath}
  \Big|\frac{1}{|G_N|} \E \tr Q((P_{S^{N,i}})_{i\le 2m}) - \frac{1}{|G_N|} \E \tr Q((P_{Z^{N,i}})_{i\le 2m})\Big| \to 0
\end{displaymath}
as $N \to \infty$ or equivalently, after passing to Fourier transforms via Corollary \ref{cor:deterministic-spectrum-explained}, that
\begin{align}\label{eq:goal-goal}
  \Big|\sum_{\Lambda \in \widehat{G_N}} \mu_{G_N}(\Lambda) \frac{1}{\dim \Lambda} \Big(\E \tr \widehat{S^{N,i_1}}(\Lambda)\cdots \widehat{S^{N,i_n}}(\Lambda) - \E \tr \widehat{Z^{N,i_1}}(\Lambda)\cdots \widehat{Z^{N,i_n}}(\Lambda)\Big)\Big| \to 0.
\end{align}

\medskip

We will first show that for every $\varepsilon$ there exists $N_0$ such that for all $N > N_0$ and any representation $\Lambda$ of $G_N$
\begin{align}\label{eq:high-dim-epsilon}
  \Big|\frac{1}{\dim \Lambda} \Big(\E \tr \widehat{S^{N,i_1}}(\Lambda)\cdots \widehat{S^{N,i_n}}(\Lambda) - \E \tr \widehat{Z^{N,i_1}}(\Lambda)\cdots \widehat{Z^{N,i_n}}(\Lambda)\Big)\Big| \le \varepsilon.
\end{align}

For this it is enough to show that for any sequence $G_N$ such that $|G_N|\to \infty$  and any sequence of irreducible representations $\Lambda_N \in \widehat{G}_N$ of dimension $d_N$ we have
\begin{align}\label{eq:high-dim-limes}
  \Big |\frac{1}{d_N} \Big(\E \tr \widehat{S^{N,i_1}}(\Lambda_N)\cdots \widehat{S^{N,i_n}}(\Lambda_N) - \E \tr \widehat{Z^{N,i_1}}(\Lambda_N)\cdots \widehat{Z^{N,i_n}}(\Lambda_N)\Big)\Big| \to 0,
\end{align}
since if \eqref{eq:high-dim-epsilon} does not hold, than we can choose a subsequence $G_{N_k}, \Lambda_{N_k}$ for which \eqref{eq:high-dim-limes} fails.

We will now show that for every $\xi$ such that $\E \xi = \E \xi^2 = 0$, $\E |\xi|^2 = 1$,
\begin{displaymath}
 M_N := \frac{1}{d_N} \E \tr \widehat{S^{N,i_1}}(\Lambda_N)\cdots \widehat{S^{N,i_n}}(\Lambda_N) = M_N' + o(1),
\end{displaymath}
where $M_N'$ depends only on $N$ but not on $\xi$. Since the standard complex Gaussian variable satisfies the above moment conditions, this will end the proof of \eqref{eq:high-dim-limes} and as a consequence of \eqref{eq:high-dim-epsilon}.

Let us define $\widetilde{S}^{N,i}_g = S^{N,i}_g = \frac{1}{\sqrt{|G_N|}} X^{N,(i+1)/2}_g$ for $i$ odd and $\widetilde{S}^{N,i}_g = S^{N,i}_{g^{-1}} = \frac{1}{\sqrt{|G_N|}} \overline{X^{N,i/2}_g}$ for $i$ even.
Denote also $\Lambda_N(g,i) = \Lambda_N(g)$ for $i$ odd and $\Lambda_N(g,i) = \Lambda_N(g)^\ast$ for $i$ even.
With this notation we have
\begin{align}\label{eq:fourier-even-odd}
  \widehat{S^{N,i}}(\Lambda_N) = \sum_{g \in G_N} \widetilde{S}^{N,i}_g \Lambda_N(g,i)
\end{align}
and so we have
\begin{displaymath}
  M_N = \sum_{g_1,\ldots g_{n} \in G_N} \frac{1}{d_N} (\E \prod_{l=1}^{n} \widetilde{S}^{N,i_l}_{g_l})\tr \prod_{l=1}^{n} \Lambda_N(g_l,i_l).
\end{displaymath}

If there exists $h \in G_N$ such that $g_i = h$ for exactly one $i \in \{1,\ldots,n\}$ then due to independence of the variables $X^{N,i}_g$, $i \le m, g \in G_N$, we have $\E \prod_{l=1}^{n} \widetilde{S}^{N,i_l}_{g_l} = 0$, so the sequence $(g_l)_{l=1}^{n}$ does not contribute to $M_N$.

On the other hand the cardinality of the set of sequences $(g_l)_{l=1}^{n}$ such that for every $h \in G_N$, $g_i = h$ for none or at least two indices $i$ and for some $h \in G_N$ there are more than two indices with this property, is bounded by $C_n |G_N|^{\lceil n/2 \rceil - 1}$ for some constant $C_n$ depending only on $n$. Moreover, by H\"older's inequality we have $|\E \prod_{l=1}^{n} \widetilde{S}^{N,i_l}_{g_l}| \le D_n |G_N|^{-n/2}$ for some constant $D_n$ depending on $n$ and $\xi$. Finally  $ |\tr \prod_{l=1}^n \Lambda_N(g_l,i_l)| \le d_N$ since $\Lambda_N$ is a unitary representation.

Thus the contribution to $M_N$ from this class of sequences is at most
\begin{displaymath}
  C_n |G_N|^{\lceil n/2\rceil -1} \frac{1}{d_N} D_n |G_N|^{-n/2} d_N \le  \frac{C_nD_n}{\sqrt{|G_N|}} \stackrel{N \to \infty}{\to} 0.
\end{displaymath}

To each sequence $(g_l)_{l=1}^n$ which does not belong to the two classes considered above (note that such sequences exist only for even $n$) one can associate a pair partition $P = \{ \{j_l,k_l\}\colon l =1,\ldots,n/2\}$
of the set $\{1,\ldots,n\}$ such that for all $l$, $g_{j_l} = g_{k_{l}}$ and for $l_1\neq l_2$, $g_{j_{l_1}} \neq g_{j_{l_2}}$.
By the independence assumption
\begin{displaymath}
  \E \prod_{l=1}^{n} \widetilde{S}^{N,i_l}_{g_l} = \prod_{l=1}^{n/2} \E (\widetilde{S}^{N,i_{j_l}}_{g_{j_l}} \widetilde{S}^{N,i_{k_l}}_{g_{k_l}}).
\end{displaymath}
The expectation $\E (\widetilde{S}^{N,i_{j_l}}_{g_{j_l}} \widetilde{S}^{N,i_{k_l}}_{g_{k_l}})$  equals
\begin{displaymath}
  \begin{cases}
    \frac{1}{|G_N|}\E \xi^2 = 0  & \mbox{if } i_{j_l} = i_{k_l}\;\textrm{and $i_{j_l}$ is odd} \ \\
    \frac{1}{|G_N|}\E (\overline{\xi})^2 = 0& \mbox{if } i_{j_l} = i_{k_l} \; \textrm{and $i_{j_l}$ is even}  \\
    \frac{1}{|G_N|}\E |\xi|^2 = \frac{1}{|G_N|}& \mbox{if } |i_{j_l} - i_{k_l}| = 1 \; \textrm{and $\min\{i_{j_l},i_{k_l}\}$ is odd}  \\
    0 &\ \mbox{otherwise},
    \end{cases}
\end{displaymath}
where in the last case we used again independence (the Reader may easily verify that this case reduces to subcases in which one of the expectations $\E \xi \zeta$, $\E \bar{\xi} \zeta$, $\E \bar{\xi} \bar{\zeta}$ appears, where $\zeta$ is an independent copy of $\xi$; the precise form of the expectation depends on the parity of $i_{j_l}$ and $i_{k_l}$).
Define now $\mathcal{C}_N$ as the set of  all sequences $(g_l)_{l=1}^n$ in which every element $h \in G_N$ either appears twice or does not appear at all, and set
\begin{displaymath}
  M_N' =  \sum_{(g_1,\ldots g_{n}) \in \mathcal{C}_N} \frac{1}{d_N} (\E \prod_{l=1}^{n} \widetilde{S}^{N,i_l}_{g_l})\tr \prod_{l=1}^{2m} \Lambda_N(g_l,i_l).
\end{displaymath}
The above discussion shows that $M_N = M_N' + o(1)$ and $M_N'$ indeed does not depend on $\xi$ (satisfying the assumptions of the lemma). This ends the proof of \eqref{eq:high-dim-epsilon}.

To conclude the proof of the lemma fix $\varepsilon > 0$ and let $N_0$ be such that \eqref{eq:high-dim-epsilon} holds for all $N>N_0$. Using the fact that $\mu_{G_N}$ is a probability measure we obtain that for such $N$,
\begin{displaymath}
  \Big|\sum_{\Lambda \in \widehat{G_N}} \mu_{G_N}(\Lambda) \frac{1}{\dim \Lambda} \Big(\E \tr \widehat{S^{N,i_1}}(\Lambda)\cdots \widehat{S^{N,i_n}}(\Lambda) - \E \tr \widehat{Z^{N,i_1}}(\Lambda)\cdots \widehat{Z^{N,i_n}}(\Lambda)\Big)\Big| \le \varepsilon,
\end{displaymath}
which proves \eqref{eq:goal-goal} and ends the proof of the lemma.

\end{proof}

\section{Central Limit Theorem for the Gaussian case}\label{sec:CLT}

In this section we will consider Central Limit Theorems for linear eigenvalue statistics of $\frac{1}{|G_N|} P_{X^N}P_{X^N}^\ast$ where $X^N = (X^N_g)_{g \in G_N}$ is a family of standard complex Gaussian variables. Contrary to the previous sections, we will consider squares of singular values of $P_{X^N}$ rather then the singular values themselves. We choose this version since it is more common in the literature concerning Central Limit Theorems  and the formulation of results is less involved. As a starting point, let us recall the well known Central Limit Theorem for linear eigenvalue statistics of the Laguerre ensemble (see e.g. \cite[Theorem 7.3.1. and Remark 7.3.2]{MR2808038}).

\begin{theorem}\label{thm:CLT-Laguerre}
  Let $f\colon \R \to \R$ be a $C^1$, Lipschitz function and let $\Gamma_n$ be an $n\times n$ complex Ginibre matrix (normalized to have variance of entries equal to $1/n$). Then, as $n \to \infty$, the random variable
  \begin{displaymath}
    Z_n = \tr f(\Gamma_n \Gamma_n^\ast) - \E \tr f(\Gamma_n \Gamma_n^\ast)
  \end{displaymath}
  converges in distribution to the Gaussian variable with variance
  \begin{align}\label{eq:V-infinity}
    V_\infty = \frac{1}{4\pi^2} \int_0^4 \int_0^4 \Big(\frac{f(\lambda_1)- f(\lambda_2)}{\lambda_1 - \lambda_2}\Big)^2 \frac{4-(\lambda_1-2)(\lambda_2-2)}{\sqrt{4- (\lambda_1-2)^2}\sqrt{4 - (\lambda_2-2)^2}}d\lambda_1d\lambda_2.
  \end{align}
\end{theorem}

The main result of this section is
\begin{theorem}\label{thm:CLT-Gaussian}
 Let $G_N$ be a sequence of finite groups with $|G_N|\to \infty$ and let $X^N = (X^N_g)_{g \in G_N}$ be i.i.d. standard Gaussian variables and $P_{X^N}$ be the corresponding random convolution operator. Assume that the measures $\widetilde{\mu}_N$ converge weakly to some probability measure $\mu$ on $\overline{\Z}_+$. Let $f \colon \R \to \R$ be a $C^1$, Lipschitz function. For $n \in \Z_+$ define
 \begin{align}\label{eq:V-n}
   V_n = \Var (\tr f(\Gamma_n\Gamma_n^\ast)),
 \end{align}
 where $\Gamma_n$ is an $n\times n$ Ginibre matrix. Define also the random variable
 \begin{displaymath}
 S_N = \frac{1}{\sqrt{|G_N|}}\tr f\Big(\frac{1}{|G_N|}P_{X^N}P_{X^N}^\ast\Big).
 \end{displaymath}
 Then, as $N \to \infty$,
 the random variable
 \begin{displaymath}
   \widetilde{S}_N = S_N-\E S_N
 \end{displaymath}
 converges in distribution to the Gaussian variable with mean zero and variance
 \begin{displaymath}
   \sigma^2 = \sum_{n \in \overline{\Z}_+} \mu(n) V_n.
 \end{displaymath}
\end{theorem}

\begin{remark} If $\mu$ is the Dirac's mass at infinity we get the same limiting distribution as in Theorem \ref{thm:CLT-Laguerre}. Note however that in the case of random convolutions we need to normalize the linear eigenvalue statistic with the usual CLT normalization, which is in contrast with the behavior for sample covariance matrices.
\end{remark}

To prove Theorem \ref{thm:CLT-Gaussian} we will need a simple lemma, which is a part of folklore. However, since we were not able to find it in the literature, we provide the proof. We note that it relies on a version of the Poincar\'e inequality for Gaussian measures which is an important ingredient in the proof of the CLT for sample covariance matrices (Theorem \ref{thm:CLT-Laguerre}).

\begin{lemma}\label{le:4th-moment}
Under the notation of Theorem \ref{thm:CLT-Laguerre}, there exists a universal constant $C$, such that for all $n \in \overline{\Z}_+$,
\begin{displaymath}
  \E |Z_n|^4 \le C \sup_{x \in \R} |f'(x)|^4.
\end{displaymath}
\end{lemma}

Before we present the proof of the above lemma, let us state an immediate corollary to it and Theorem \ref{thm:CLT-Laguerre}.
\begin{corollary}\label{cor:convergence-of-variances}
Let $V_\infty$ and $V_n$ be defined by \eqref{eq:V-infinity} and \eqref{eq:V-n} respectively. Then
\begin{displaymath}
  \lim_{n\to \infty} V_n = V_\infty.
\end{displaymath}
\end{corollary}

\begin{proof}[Proof of Lemma \ref{le:4th-moment}]
Let $\lambda_1\le\ldots\le\lambda_n$ be the singular values of $\Gamma_n$. By the Hoffman-Wielandt inequality for singular values (see \cite[Theorem A.37 (ii)]{MR2567175}) the map $\Gamma_n \to (\lambda_1,\ldots,\lambda_n)$ is 1-Lipschitz (where the space of matrices is endowed with the Hilbert-Schmidt norm). Since $\sqrt{n}\Gamma_n$ can be seen as a Gaussian random vector in  dimension $2n^2$, with covariance matrix equal to one half of the identity, from the Poincar\'e inequality for the Gaussian measure  (see e.g. \cite[Theorem 3.20]{MR3185193}), the random vector $\lambda = (\lambda_1,\ldots,\lambda_n)$ satisfies the Poincar\'e inequality with constant $K/n$ for some universal constant $K$, i.e. for any smooth function $g\colon \R^n \to \R$,
\begin{displaymath}
  \Var g(\lambda) \le \frac{K}{n}\E |\nabla g(\lambda)|^2.
\end{displaymath}
It is well known (see e.g. \cite[Proposition 2.5 and Lemma 2.1]{MR2507637}) that this implies
\begin{displaymath}
  \E |g(\lambda)- \E g(\lambda)|^4 \le \frac{K'}{n^2}\E |\nabla g(\lambda)|^4
\end{displaymath}
for some constant $K'$, depending only on $K$.

Since $Z_n = \sum_{i=1}^n (f(\lambda_i^2) -\E f(\lambda_i^2))$, we obtain
\begin{displaymath}
\E |Z_n|^4 \le \frac{K'}{n^2} \E |\sum_{i=1}^n 4 f'(\lambda_i^2)^2 \lambda_i^2|^2 \le 16 \frac{K'}{n^2} \|f'\|_\infty^4 \E g(\lambda)^4,
\end{displaymath}
where $g(\lambda) = (\sum_{i=1}^n \lambda_i^2)^{1/2}$. The function $g$  is 1-Lipschitz, therefore
\begin{displaymath}
\E |g(\lambda) - \E g(\lambda)|^4 \le \frac{K'}{n^2}.
\end{displaymath}
Moreover $\E g(\lambda) \le (\sum_{i=1}^n \E \lambda_i^2)^{1/2} = (\E \|\Gamma_n \|_{HS}^2)^{1/2} = \sqrt{n}$.
Combining the last three estimates with the triangle inequality in $L_4$ we arrive at the assertion of the lemma.
\end{proof}

We are now ready for
\begin{proof}[Proof of Theorem \ref{thm:CLT-Gaussian}]
  By Proposition \ref{prop:Gaussian-distribution} the random variable $S_N$ has the same distribution as
  \begin{displaymath}
    \sum_{\Lambda \in \widehat{G}_N} \frac{\dim \Lambda}{\sqrt{|G_N|}} \tr f(\Gamma_\Lambda \Gamma_\Lambda^\ast),
  \end{displaymath}
  therefore in what follows we may and we will identify these variables.

In particular,  thanks to independence of $\Gamma_\Lambda$ and the definitions of measures $\mu_{G_N}$ and $\widetilde{\mu}_{G_N}$, we immediately get

\begin{align*}
\Var S_N = \sum_{n \in \Z_+} \widetilde{\mu}_{G_N}(n) V_n,
\end{align*}
and the convergence of $\widetilde{\mu}_{G_N}$ to $\mu$, together with Corollary \ref{cor:convergence-of-variances} imply that
\begin{align}\label{eq:variance-convergence}
  \lim_{N \to \infty} \Var S_N = \lim_{N \to \infty} \sum_{n \in \Z_+} \widetilde{\mu}_{G_N}(n) V_n = \sigma^2.
\end{align}

To shorten the notation denote
\begin{displaymath}
T_N(\Lambda)  = \frac{\dim \Lambda}{\sqrt{|G_N|}} \Big(\tr f(\Gamma_\Lambda \Gamma_\Lambda^\ast) - \E \tr f(\Gamma_\Lambda \Gamma_\Lambda^\ast)\Big).
\end{displaymath}

For fixed $\varepsilon > 0$ let $I_N(\varepsilon) = \{\Lambda \in \widehat{G}_N\colon \dim \Lambda \ge \varepsilon \sqrt{|G_N|}\}$. Note that the cardinality of
$I_N$ is at most $1/\varepsilon^2$. By Theorem \ref{thm:CLT-Laguerre}, we have
\begin{displaymath}
\E \exp\Big(\sqrt{-1}t \Big(\tr f(\Gamma_n \Gamma_n^\ast) - \E \tr f(\Gamma_n\Gamma_n^\ast)\Big)\Big) \to e^{-t^2V_\infty /2},
\end{displaymath}
uniformly with respect to $t$ on compact sets. By Corollary \ref{cor:convergence-of-variances} we also have
\begin{displaymath}
  e^{-t^2V_n/2} \to e^{-t^2V_\infty/2}
\end{displaymath}
uniformly over $t$.

Thus, using the estimate $|\prod_{i=1}^k a_i - \prod_{i=1}^k b_i| \le \sum_{i=1}^k |a_i-b_i|$ for $|a_i|,|b_i| \le 1$, we get
\begin{displaymath}
\Big|\E e^{\sqrt{-1} t \sum_{\Lambda \in I_N(\varepsilon)} T_N(\Lambda)} - e^{-2^{-1}t^2 \sum_{n \ge \varepsilon \sqrt{|G_N|}} \widetilde{\mu}_{G_N}(n) V_{n}}\Big| \to 0
\end{displaymath}
as $N \to \infty$ for any fixed $\varepsilon > 0$ (note that if $I_N(\varepsilon)$ is empty then the left-hand side above vanishes). One can therefore find a sequence $\varepsilon_N \to 0$, such that
\begin{align}\label{eq:CLT-large-irreps}
\Big|\E e^{\sqrt{-1} t \sum_{\Lambda \in I_N(\varepsilon_N)} T_N(\Lambda)} - e^{-2^{-1}t^2\sum_{n \ge \varepsilon_N \sqrt{|G_N|}} \widetilde{\mu}_{G_N}(n) V_{n}}\Big| \to 0.
\end{align}

We will now show that for \emph{every} sequence $\varepsilon_N \to 0$, and any $t \in \R$,

\begin{align}\label{eq:CLT-small-irreps}
\Big|\E e^{\sqrt{-1} t \sum_{\Lambda \notin I_N(\varepsilon_N)} T_N(\Lambda)} - e^{-2^{-1}t^2\sum_{n < \varepsilon_N \sqrt{|G_N|}} \widetilde{\mu}_{G_N}(n) V_{n}}\Big| \to 0.
\end{align}

This will end the proof of the theorem, because together with \eqref{eq:CLT-large-irreps} and independence of $\Gamma_\Lambda$, \eqref{eq:CLT-small-irreps} will yield
\begin{displaymath}
 \Big |\E e^{\sqrt{-1} t (S_N - \E S_N)} -  e^{-2^{-1}t^2\sum_{n\in \Z_+} \widetilde{\mu}_{G_N}(n)V_n}\Big| \to 0,
\end{displaymath}
which by \eqref{eq:variance-convergence} gives
\begin{displaymath}
  \E e^{\sqrt{-1} t (S_N - \E S_N)} = e^{-t^2\sigma^2/2}.
\end{displaymath}

To prove \eqref{eq:CLT-small-irreps} it is enough to show that it holds under an additional assumption that there exists a limit
\begin{align}\label{eq:technical-convergence-assumption}
  \gamma^2 = \lim_{N \to \infty} \sum_{n < \varepsilon_N\sqrt{|G_N|}} \widetilde{\mu}_{G_N}(n) V_n.
\end{align}
Indeed, the sequence $\sum_{n < \varepsilon_N \sqrt{|G_N|}} \widetilde{\mu}_{G_N}(n) V_n$ is bounded, so if \eqref{eq:CLT-small-irreps} holds under the additional assumption \eqref{eq:technical-convergence-assumption},
then from every subsequence of the original sequence
\begin{displaymath}
\Big|\E e^{\sqrt{-1} t \sum_{\Lambda \notin I_N(\varepsilon_N)} T_N(\Lambda)} - e^{-2^{-1}t^2\sum_{n < \varepsilon_N\sqrt{|G_N|}} \widetilde{\mu}_{G_N}(n) V_{n}}\Big|\end{displaymath}
one can choose another subsequence, which converges to zero. This implies \eqref{eq:CLT-small-irreps}.

We will thus assume \eqref{eq:technical-convergence-assumption} and show that
\begin{align}\label{eq:final-convergence}
\lim_{N\to \infty} \E e^{\sqrt{-1} t  \sum_{\Lambda \notin I_N(\varepsilon_N)} T_N(\Lambda)} = e^{-t^2\gamma^2/2},
\end{align}
which clearly implies \eqref{eq:CLT-small-irreps}.
The convergence \eqref{eq:final-convergence} is a consequence of just the usual Lindeberg Central Limit Theorem. Indeed, consider the triangular array
$(T_{N,\Lambda})_{N\in \overline{\Z}_+,\Lambda \in \widehat{G_N}\setminus I_N(\varepsilon_N)}$. We have $\E T_{N,\Lambda} = 0$, $\Var(\sum_{\Lambda \notin I_N(\varepsilon_N)} T_{N,\Lambda}) \to \gamma^2$ and it remains to check the Lindeberg condition. We have for any $\delta > 0$,
\begin{align*}
&Lind_N (\delta) :=\sum_{\Lambda \notin I_N(\varepsilon_N)} \E T_{N,\Lambda}^2\ind{\{|T_{N,\Lambda}|>\delta\}} \\
&=
\sum_{n <  \varepsilon_N \sqrt{|G_N|}} \widetilde{\mu}_{G_N}(n)\E |\tr f(\Gamma_n\Gamma_n^\ast) - \E \tr f(\Gamma_n\Gamma_n^\ast)|^2
\ind{\{n|G_N|^{-1/2}|\tr f(\Gamma_n\Gamma_n^\ast) - \E \tr f(\Gamma_n\Gamma_n^\ast)| > \delta\}}.
\end{align*}
We clearly have
\begin{align*}
Lind_N(\delta) \le & \frac{1}{\delta^2 } \sum_{n <  \varepsilon_N \sqrt{|G_N|}}\frac{n^2}{|G_N|} \widetilde{\mu}_{G_N}(n) \E |\tr f(\Gamma_n\Gamma_n^\ast) - \E \tr f(\Gamma_n\Gamma_n^\ast)|^4 \\
\le & \frac{1}{\delta^2 } \sum_{n <  \varepsilon_N \sqrt{|G_N|}}\frac{n^2}{|G_N|} \widetilde{\mu}_{G_N}(n) C\|f'\|_\infty^4\\
\le & \frac{C\|f'\|_\infty^4 \varepsilon_N^2}{\delta^2 } \stackrel{N\to \infty}{\to} 0,
\end{align*}
where in the second inequality we used Lemma \ref{le:4th-moment} and in the last one the fact that $\widetilde{\mu}_{G_N}$ is a probability measure.
This ends the proof of Theorem \ref{thm:CLT-Gaussian}.
\end{proof}

\section{Further questions}\label{sec:questions}

As shown in the preceding sections, random convolution operators on large finite groups exhibit a behavior, which combines features from the theory of large random matrices with independent entries, with those of classical probability theory. In our opinion, this is one of the reasons making this model interesting and worth investigating. Clearly there are multiple other aspects of the theory of random matrices, which we have not touched in the present work and which may potentially lead to interesting phenomena. Also, even the basic results obtained above give rise to new questions, both from the probabilistic and representation theory points of view. Below we list some of them.

\begin{enumerate}[(I)]
\item As already mentioned in Section \ref{sec:examples} we do not know how to characterize the class of measures $\mu$ on $\overline{\Z}_+$, which are weak limits of projected Plancherel measures. In fact the class of examples we presented contains only finite mixtures. It would be of interest to see whether one can obtain more complicated examples.
\medskip

\item While the limiting singular-value distribution follows a universal behaviour (as shown in Theorem \ref{thm:singular-values}), the eigenvalue distribution has been analysed in Theorem \ref{thm:eigenvalues} only for Gaussian matrices. It is natural to conjecture universality also in this case. For sequences of groups with representations of bounded degrees this can be done by following the ideas of Meckes (which correspond to the small-dimensional representations part of our proof). For general groups we encounter a similar problem of instability of the spectral measure as in the classical circular law, moreover the model of random convolutions involves less independence (it is defined in terms of a smaller number of independent variables) than Wigner-like matrices. On the one hand, due to additive structure (Fourier transforms are sums of randomly weighted unitary matrices) this resembles somehow the questions related to the analysis of directed $d$-regular graphs or sums of random permutation or random unitary matrices, which have attracted considerable attention in recent years and are known to be difficult (see \cite{MR3091727,2017arXiv170509053B, 2017arXiv170305839C}). On the other hand, the algebraic structure of the model, may give additional advantage in the proofs. In any case, it seems that the analysis should be related to the problem of bounding the smallest singular value of the matrices in question.

    Narrowing the focus from the most general case to a more specific one,  we can formulate the following question, for what we believe is the most interesting example:

    \medskip
    \emph{Is there universality of the limiting eigenvalue distribution for random convolution operators on the symmetric group $S_N$, with $N \to \infty$?}

\medskip
\item For Abelian groups, Meckes allows more general covariance structure that the one we dealt with, namely he considers the random variable $\xi$ in Theorems \ref{thm:singular-values}, \ref{thm:eigenvalues}, which satisfies $\E \xi^2 = \alpha$ for some $\alpha \in [0,1]$. This leads to interesting phenomena, since the limiting eigenvalue distribution turns out to be a mixture of Gaussian distributions,  governed by the limiting density of the set of elements of order two in $G_N$. More specifically under the assumption of the existence of the limit $p = \lim_{N\to \infty} \frac{|\{a\in G_N\colon a^2 = 1\}|}{|G_N|}$ Meckes proved that the limiting spectral distribution of $\frac{1}{\sqrt{|G_N|}}P_{X^N}$ equals $(1-p)\gamma_{0} + p\gamma_\alpha$, where $\gamma_{\alpha}$ is the centered Gaussian measure on $\C \simeq \R^2$ with the covariance matrix equal to
    \begin{displaymath}
      \frac{1}{2}\left[\begin{array}{cc}
        1+\alpha & 0  \\
        0 & 1-\alpha
      \end{array}\right]
    \end{displaymath}
    (in particular $\gamma_0$ is the standard complex Gaussian distribution and $\gamma_1$ the standard real Gaussian distribution).

    It would be interesting to go beyond the assumption $\E \xi^2 = 0$ also in the general non-Abelian case.
\medskip
\item One can also ask about extensions of the Central Limit Theorem to more general distributions of the entries. We refer to \cite{MR2567175, MR2561434, MR2808038} for results concerning limit theorems for sample covariance matrices based on rectangular matrices with general independent entries.
\medskip
\item Another natural question is the behaviour of the operator norm of random convolution operators. See \cite{MR2795050} for results concerning the operator norm in the case $G_N = \Z_N$.
\medskip
\item  One can also ask about the local behaviour in the bulk of the spectrum, especially in the case, when high-dimensional representations dominate. On the one hand one could expect a sine kernel type behaviour, as in Wigner matrices, on the other hand the fact that (at least in the Gaussian case) the spectrum receives contribution from many independent random matrices, some Poisson type behaviour may also be expected.

\item There are also many interesting aspects related to the study of convolutions with random class functions (i.e. functions constant on conjugacy classes of the group), which have been recently studied by M. Meckes \cite{MarkPersonalCommunication}.
    \end{enumerate}
\appendix

\section{Auxiliary lemmas on weak convergence in probability}

We gather here some standard facts concerning weak convergence in probability, which may however be difficult to find in the literature in the precise form needed for our purposes. Since the proofs rely on standard analytic arguments, we will present only their sketches.

\begin{lemma}\label{le:convergence-functions}
Let $\nu_N$ be a sequence of random Borel probability measures on $\C$ and let $\nu$ be a deterministic Borel probability measure on $\C$. Then the following statements are equivalent
\begin{itemize}
\item[(i)] $\nu_N$ converges to $\nu$ weakly in probability
\item[(ii)] for every bounded continuous function $f\colon \C \to \R$, the sequence of real random variables $\int_\C f d\nu_N$ converges in probability to $\int_\C f d\nu$.
\end{itemize}
\end{lemma}

\begin{proof}
  Assume (i), and note that from any subsequence of $\nu_N$ we can select a subsequence $\nu_{N_m}$ such that $d(\nu_{N_m},\nu) \to 0$ almost surely. Thus the corresponding subsequence $\int_\C f d\nu_{N_m}$ converges almost surely to $\int_\C f d\nu$. This however implies that the whole sequence $\int_\C f d\nu_N$ converges to $\int fd\nu$ in probability.

  To prove that (ii) implies (i) consider a countable set $\mathcal{A}$ of compactly supported continuous functions on $\C$ such that any compactly supported continuous function can be uniformly approximated by elements of $\mathcal{A}$. If (ii) is satisfied then for any subsequence of the sequence $\nu_N$, by using Cantor's diagonal argument we can select a subsequence $\nu_{N_m}$ such that with probability one for all  $f \in \mathcal{A}$ simultaneously $\int_\C f d\nu_{N_m} \to \int_{\C} f d\nu$. Let $\Omega'$ be the event of full measure on which this convergence holds. By the definition of $\mathcal{A}$ on $\Omega'$ we have $\int_\C f d\nu_{N_m} \to \int_\C f d\nu$ for all compactly supported continuous functions $f$, which implies that on $\mathcal{A}$ the subsequence $\nu_{N_m}$ converges weakly to $\nu$. Thus every subsequence of the sequence $\nu_N$ contains another subsequence converging to $\nu$ almost surely, which implies that $\nu_N$ converges to $\nu$ in probability.
\end{proof}

\begin{remark} It is not difficult to provide a more direct proof of the above lemma, e.g. by analyzing the base of topology for the weak convergence of probability measures or by introducing a distance, which metrizes the weak convergence and depends only on a countable family of functions (see e.g. the proof of Theorem 2.4.4. in \cite{MR2760897} for a similar derivation). We chose the argument based on the diagonal method since it seems to be slightly shorter to sketch and it allows for a more uniform treatment of the above lemma and the following one.
\end{remark}

\begin{lemma}\label{le:convergence-moments}
Let $\nu_N$ be a sequence of random probability measures on $\R$, such that with probability one $\nu_N$ has finite moments of all orders and let $\nu$ be a deterministic probability measure on $\R$, determined by its moments. If for all positive integers $k$,
\begin{displaymath}
  \int_\R x^k d\nu_N \to \int_\R x^k d\nu
\end{displaymath}
in probability, then
$\nu_N$ converges to $\nu$ weakly in probability.
\end{lemma}

\begin{proof}
  The argument is similar to the proof of the implication from (ii) to (i) in the previous lemma.
  It is enough to show that from every subsequence of the sequence $\nu_N$ one can select another subsequence converging to $\nu$ almost surely. This can however be easily obtained by Cantor's diagonal argument, since for every $k$ from each subsequence of $\nu_N$ one can select a subsequence $\nu_{N_m}$ such that $\int_\R x^k d\nu_{N_m} \to \int_\R x^k d\nu$ almost surely as $m \to \infty$. The final subsequence obtained by the diagonal method has the property that with probability one all its moments converge to the corresponding moments of $\nu$, and since $\nu$ is uniquely determined by its moments this implies that the subsequence indeed converges to $\nu$ almost surely.
\end{proof}

\bibliographystyle{amsplain}	

\begin{thebibliography}{10}

\bibitem{MR2760897}
Greg~W. Anderson, Alice Guionnet, and Ofer Zeitouni, \emph{An introduction to
  random matrices}, Cambridge Studies in Advanced Mathematics, vol. 118,
  Cambridge University Press, Cambridge, 2010. \MR{2760897}

\bibitem{MR1428519}
Z.~D. Bai, \emph{Circular law}, Ann. Probab. \textbf{25} (1997), no.~1,
  494--529. \MR{1428519}

\bibitem{MR1711663}
\bysame, \emph{Methodologies in spectral analysis of large-dimensional random
  matrices, a review}, Statist. Sinica \textbf{9} (1999), no.~3, 611--677, With
  comments by G. J. Rodgers and Jack W. Silverstein; and a rejoinder by the
  author. \MR{1711663}

\bibitem{MR2567175}
Zhidong Bai and Jack~W. Silverstein, \emph{Spectral analysis of large
  dimensional random matrices}, second ed., Springer Series in Statistics,
  Springer, New York, 2010. \MR{2567175}

\bibitem{2017arXiv170509053B}
A.~{Basak}, N.~{Cook}, and O.~{Zeitouni}, \emph{{Circular law for the sum of
  random permutation matrices}}, ArXiv e-prints (2017).

\bibitem{MR3091727}
Anirban Basak and Amir Dembo, \emph{Limiting spectral distribution of sums of
  unitary and orthogonal matrices}, Electron. Commun. Probab. \textbf{18}
  (2013), no. 69, 19. \MR{3091727}

\bibitem{MR2908617}
Charles Bordenave and Djalil Chafa\"\i, \emph{Around the circular law}, Probab.
  Surv. \textbf{9} (2012), 1--89. \MR{2908617}

\bibitem{MR2682263}
Arup Bose, Sreela Gangopadhyay, and Arnab Sen, \emph{Limiting spectral
  distribution of {$XX'$} matrices}, Ann. Inst. Henri Poincar\'e Probab. Stat.
  \textbf{46} (2010), no.~3, 677--707. \MR{2682263}

\bibitem{MR2795050}
Arup Bose, Rajat~Subhra Hazra, and Koushik Saha, \emph{Spectral norm of
  circulant-type matrices}, J. Theoret. Probab. \textbf{24} (2011), no.~2,
  479--516. \MR{2795050}

\bibitem{MR1945684}
Arup Bose and Joydip Mitra, \emph{Limiting spectral distribution of a special
  circulant}, Statist. Probab. Lett. \textbf{60} (2002), no.~1, 111--120.
  \MR{1945684}

\bibitem{MR2399292}
Arup Bose and Arnab Sen, \emph{Another look at the moment method for large
  dimensional random matrices}, Electron. J. Probab. \textbf{13} (2008), no.
  21, 588--628. \MR{2399292}

\bibitem{MR2827968}
Arup Bose, Rajat Subhra~Hazra, and Koushik Saha, \emph{Patterned random
  matrices and method of moments}, Proceedings of the {I}nternational
  {C}ongress of {M}athematicians. {V}olume {IV}, Hindustan Book Agency, New
  Delhi, 2010, pp.~2203--2231. \MR{2827968}

\bibitem{MR3185193}
St\'ephane Boucheron, G\'abor Lugosi, and Pascal Massart, \emph{Concentration
  inequalities}, Oxford University Press, Oxford, 2013, A nonasymptotic theory
  of independence, With a foreword by Michel Ledoux. \MR{3185193}

\bibitem{MR2206341}
W{\l}odzimierz Bryc, Amir Dembo, and Tiefeng Jiang, \emph{Spectral measure of
  large random {H}ankel, {M}arkov and {T}oeplitz matrices}, Ann. Probab.
  \textbf{34} (2006), no.~1, 1--38. \MR{2206341}

\bibitem{2017arXiv170305839C}
N.~A. {Cook}, \emph{{The circular law for random regular digraphs}}, ArXiv
  e-prints (2017).

\bibitem{MR543191}
Philip~J. Davis, \emph{Circulant matrices}, John Wiley \&\ Sons, New
  York-Chichester-Brisbane, 1979, A Wiley-Interscience Publication, Pure and
  Applied Mathematics. \MR{543191}

\bibitem{MR964069}
Persi Diaconis, \emph{Group representations in probability and statistics},
  Institute of Mathematical Statistics Lecture Notes---Monograph Series,
  vol.~11, Institute of Mathematical Statistics, Hayward, CA, 1988. \MR{964069}

\bibitem{MR1059483}
\bysame, \emph{Patterned matrices}, Matrix theory and applications ({P}hoenix,
  {AZ}, 1989), Proc. Sympos. Appl. Math., vol.~40, Amer. Math. Soc.,
  Providence, RI, 1990, pp.~37--58. \MR{1059483}

\bibitem{MR1437734}
Alan Edelman, \emph{The probability that a random real {G}aussian matrix has
  {$k$} real eigenvalues, related distributions, and the circular law}, J.
  Multivariate Anal. \textbf{60} (1997), no.~2, 203--232. \MR{1437734}

\bibitem{MR0173726}
Jean Ginibre, \emph{Statistical ensembles of complex, quaternion, and real
  matrices}, J. Mathematical Phys. \textbf{6} (1965), 440--449. \MR{0173726}

\bibitem{MR773436}
V.~L. Girko, \emph{The circular law}, Teor. Veroyatnost. i Primenen.
  \textbf{29} (1984), no.~4, 669--679. \MR{773436}

\bibitem{MR2663633}
Friedrich G\"otze and Alexander Tikhomirov, \emph{The circular law for random
  matrices}, Ann. Probab. \textbf{38} (2010), no.~4, 1444--1491. \MR{2663633}

\bibitem{MR2167641}
Christopher Hammond and Steven~J. Miller, \emph{Distribution of eigenvalues for
  the ensemble of real symmetric {T}oeplitz matrices}, J. Theoret. Probab.
  \textbf{18} (2005), no.~3, 537--566. \MR{2167641}

\bibitem{MR1746976}
Fumio Hiai and D\'enes Petz, \emph{The semicircle law, free random variables
  and entropy}, Mathematical Surveys and Monographs, vol.~77, American
  Mathematical Society, Providence, RI, 2000. \MR{1746976}

\bibitem{MR832183}
Roger~A. Horn and Charles~R. Johnson, \emph{Matrix analysis}, Cambridge
  University Press, Cambridge, 1985. \MR{832183}

\bibitem{MR2561434}
A.~Lytova and L.~Pastur, \emph{Central limit theorem for linear eigenvalue
  statistics of random matrices with independent entries}, Ann. Probab.
  \textbf{37} (2009), no.~5, 1778--1840. \MR{2561434}

\bibitem{MR0208649}
V.~A. Marchenko and L.~A. Pastur, \emph{Distribution of eigenvalues in certain
  sets of random matrices}, Mat. Sb. (N.S.) \textbf{72 (114)} (1967), 507--536.
  \MR{0208649}

\bibitem{MarkPersonalCommunication}
Mark~W. Meckes, \emph{Personal communication}.

\bibitem{MR2797949}
\bysame, \emph{Some results on random circulant matrices}, High dimensional
  probability {V}: the {L}uminy volume, Inst. Math. Stat. (IMS) Collect.,
  vol.~5, Inst. Math. Statist., Beachwood, OH, 2009, pp.~213--223. \MR{2797949}

\bibitem{MR3069372}
\bysame, \emph{The spectra of random abelian {$G$}-circulant matrices}, ALEA
  Lat. Am. J. Probab. Math. Stat. \textbf{9} (2012), no.~2, 435--450.
  \MR{3069372}

\bibitem{MR2869165}
Mark~W. Meckes and Stanis\l aw~J. Szarek, \emph{Concentration for
  noncommutative polynomials in random matrices}, Proc. Amer. Math. Soc.
  \textbf{140} (2012), no.~5, 1803--1813. \MR{2869165}

\bibitem{MR1083764}
Madan~Lal Mehta, \emph{Random matrices}, second ed., Academic Press, Inc.,
  Boston, MA, 1991. \MR{1083764}

\bibitem{MR2507637}
Emanuel Milman, \emph{On the role of convexity in isoperimetry, spectral gap
  and concentration}, Invent. Math. \textbf{177} (2009), no.~1, 1--43.
  \MR{2507637}

\bibitem{MR3585560}
James~A. Mingo and Roland Speicher, \emph{Free probability and random
  matrices}, Fields Institute Monographs, vol.~35, Springer, New York; Fields
  Institute for Research in Mathematical Sciences, Toronto, ON, 2017.
  \MR{3585560}

\bibitem{MR2575411}
Guangming Pan and Wang Zhou, \emph{Circular law, extreme singular values and
  potential theory}, J. Multivariate Anal. \textbf{101} (2010), no.~3,
  645--656. \MR{2575411}

\bibitem{MR2808038}
Leonid Pastur and Mariya Shcherbina, \emph{Eigenvalue distribution of large
  random matrices}, Mathematical Surveys and Monographs, vol. 171, American
  Mathematical Society, Providence, RI, 2011. \MR{2808038}

\bibitem{MR0450380}
Jean-Pierre Serre, \emph{Linear representations of finite groups},
  Springer-Verlag, New York-Heidelberg, 1977, Translated from the second French
  edition by Leonard L. Scott, Graduate Texts in Mathematics, Vol. 42.
  \MR{0450380}

\bibitem{MR1361756}
Michel Talagrand, \emph{Concentration of measure and isoperimetric inequalities
  in product spaces}, Inst. Hautes \'Etudes Sci. Publ. Math. (1995), no.~81,
  73--205. \MR{1361756}

\bibitem{MR2906465}
Terence Tao, \emph{Topics in random matrix theory}, Graduate Studies in
  Mathematics, vol. 132, American Mathematical Society, Providence, RI, 2012.
  \MR{2906465}

\bibitem{MR2722794}
Terence Tao and Van Vu, \emph{Random matrices: universality of {ESD}s and the
  circular law}, Ann. Probab. \textbf{38} (2010), no.~5, 2023--2065, With an
  appendix by Manjunath Krishnapur. \MR{2722794}

\bibitem{MR783703}
A.~M. Vershik and S.~V. Kerov, \emph{Asymptotic behavior of the maximum and
  generic dimensions of irreducible representations of the symmetric group},
  Funktsional. Anal. i Prilozhen. \textbf{19} (1985), no.~1, 25--36, 96.
  \MR{783703}

\end{thebibliography}
\providecommand{\bysame}{\leavevmode\hbox to3em{\hrulefill}\thinspace}
\providecommand{\MR}{\relax\ifhmode\unskip\space\fi MR }
\providecommand{\MRhref}[2]{%
  \href{http://www.ams.org/mathscinet-getitem?mr=#1}{#2}
}
\providecommand{\href}[2]{#2}

\end{document}